\documentclass{article}

\usepackage{enumerate}
\usepackage{authblk}

\usepackage{amsmath, amssymb}
\usepackage{stmaryrd} 

\usepackage{hyperref}

\usepackage{tikz}

\usepackage{amsthm}

\theoremstyle{plain} 
    \newtheorem{theorem}{Theorem}[section]   
    \newtheorem{lemma}[theorem]{Lemma}
    \newtheorem{corollary}[theorem]{Corollary}
    \newtheorem{proposition}[theorem]{Proposition}
\theoremstyle{definition} 
    \newtheorem{definition}[theorem]{Definition}
    \newtheorem{example}[theorem]{Example}
    \newtheorem{notation}[theorem]{Notation}

\newcommand{\cpro}[2]{{}^{#1\rotatebox{90}{$\scriptscriptstyle{\ominus}$}}\!{#2}}
\newcommand{\rpro}[2]{{}^{#1\scriptscriptstyle{\ominus}}{#2}}

\newcommand{\sym}[1]{\mathrm{Sym}({#1})}
            
\newcommand{\gridpoints}[2]
{\foreach \y in {1,..., #1} \foreach \x in {1,..., #2}
   \fill[black] (\x,\y) circle(5pt) ; }            
            
\newcommand\extrafootertext[1]{%
    \bgroup
    \renewcommand\thefootnote{\fnsymbol{footnote}}%
    \renewcommand\thempfootnote{\fnsymbol{mpfootnote}}%
    \footnotetext[0]{#1}%
    \egroup
}            

\begin{document}

\title{Generalizations of nets and Latin squares}
\date{}
\author{Brian Curtin}
\affil{\small Department of Mathematics and Statistics,
        University of South Florida,\\
        4202 E Fowler Ave,
        Tampa, FL 33620-5700\\
        \href{mailto:bcurtin@usf.edu}{bcurtin@usf.edu}}

\maketitle


\begin{abstract}
We examine combinatorial structures that generalize  $(k,n)$-nets, orthogonal arrays, and mutually orthogonal Latin squares.
By a reticulation we mean a point set and two collections (types) of families of lines such that 
two lines of different types meet in exactly one point and each family of lines partitions the point set. 
The number of  points incident with any line depends only upon the type of the line, and every point is incident with the same number of lines of a given type.  
Each choice of a single line family of each type leads to an arrangement of  the points into a rectangular grid.  
Recording the line containing a given point in the corresponding position of an array gives a generalization of sets of mutually orthogonal Latin squares, dubbed a cooperative system.  
A cooperative system consists of a collection of column-Latin matrices and a collection of row-Latin matrices 
such that each column-Latin matrix is orthogonal to each row-Latin matrix.  
Recording lines which contain a given point as a tuple gives a generalization of certain orthogonal arrays, dubbed semi-orthogonal arrays.
Notions of parastrophy and isotopy for cooperative systems, corresponding to permuting families of lines and lines within each family, are introduced.   
Constructions of small reticulations and three recursive constructions of larger reticulations are given.  
\end{abstract}

         
\extrafootertext{\hskip-4.5ex\scriptsize  {\em 2020 Mathematics Subject Classification.} 05B15, 51E14, 05E18.}
\extrafootertext{\hskip-4.5ex\scriptsize   {\em Key words and phrases.} Bruck net, orthogonal array, Latin square, row-Latin rectangle, isotopy, parastrophy, bireversible automata.}

\section{Introduction}

We introduce combinatorial structures which can be viewed as generalizations of sets of mutually orthogonal Latin squares and the equivalent $(k,n)$-nets, orthogonal arrays, and transversal designs.    
Special cases of the structures under consideration were examined in connection with bireversible automata \cite{CurtinSavchuk:combBRA}: 
We comment on this connection after Definition \ref{def:cooppair}.  
With this application  in mind, we adapt a number of basic results from the theory to Latin squares to this new  setting.
Because they are  easily described and visualized, we begin with a point-line incidence structure analogous to a $(k,n)$-net which we call a \emph{reticulation}.  

For our purposes, a point-line incidence structure consists of 
     a set whose elements are called  \emph{points} and subsets of points called \emph{lines}.  
A point and a line are \emph{incident} when the point is an element of the line.  
By a \emph{family lines} we mean a set of lines.
Two lines are \emph{parallel} when they are equal or have empty intersection.  
The cells of any partition of the points are parallel lines.

\begin{definition}
\label{def:reticulation}
A \emph{reticulation} consists of a set of points and two (multi)sets (``types'') of families of  lines 
satisfying the following axioms. 
\begin{description}
\item[(R-1)] Two lines of different types meet in exactly one point. 
\item[(R-2)] Each family of lines partitions the set of points.
\end{description}
\end{definition}

\begin{example}
\label{ex:grid}
The prototypical reticulation consists of points arranged in an $m\times n$ array. 
The points on each horizontal line comprise one type line,  and the points on each vertical line comprise a line of the other type.   
We refer to the family of all horizontal lines as \emph{weft} and the family of vertical lines as \emph{warp}. 

\smallskip
\centering
$\mathcal{P}=\left\{ 
	\begin{tikzpicture}[scale=.5,baseline=6ex]
	   \gridpoints{3}{4};  
	\end{tikzpicture}\right\}$,
\quad 
$\mathcal{F}_{\mathrm{weft}}=\left\{
	\begin{tikzpicture}[scale=.5,baseline=6ex]
	   \foreach \y in {1,2, 3}
	         \draw[thick, red] (1,\y)--(4,\y);
	  \gridpoints{3}{4};  
	\end{tikzpicture}\right\}$,
\quad
$\mathcal{F}_{\mathrm{warp}}=\left\{
	\begin{tikzpicture}[scale=.5,baseline=6ex]
 	  \foreach \x in {1,2,3, 4}
   	      \draw[thick, blue] (\x,1)--(\x, 3);
 	 \gridpoints{3}{4};   
\end{tikzpicture}\right\}$.
\end{example}

Every finite nondegenerate reticulation can always be drawn on a rectangular grid with one type of line taking one point from each column and other type of line taking one point from each row.  
This inspired the name. We shall continue to refer to the first type as \emph{weft} and the second as \emph{warp}.  
In such a drawing, each family is drawn on a disjoint copy of the same grid (points in the same position are identified, so there is no need to draw the point set alone).

\begin{example}
\label{ex:smallrr}
We offer another small example adding families of lines to the previous example.
Here, points connected by a path lie in the same line.
 
\smallskip
\centering
$\left( \left\{
	\begin{tikzpicture}[scale=.5,baseline=6ex]
	   \foreach \x in {1,2,3, 4}. \foreach \y in {1,2, 3}
	      \draw[ red] (1,\y)--(4,\y);
	   \gridpoints{3}{4};  
	\end{tikzpicture},\,
	\begin{tikzpicture}[scale=.5,baseline=6ex]
	\draw[thick,magenta] (1,2)--(2,2) -- (3, 3) -- (4,1) ;
	\draw[magenta] (1,1)--(2,3)--(3,1)--(4,3); 
	\draw[ultra thick, magenta]  (1,3)--(2,1)--(3, 2)--(4, 2);
	\gridpoints{3}{4};  
	\end{tikzpicture}\right\}, 
\ 
\left\{
	\begin{tikzpicture}[scale=.5,baseline=6ex]
	   \foreach \x in {1,2,3, 4}. \foreach \y in {1,2, 3}
	      \draw[ blue] (\x,1)--(\x, 3); 
  	   \gridpoints{3}{4};  
	\end{tikzpicture},\,
	\begin{tikzpicture}[scale=.5,baseline=6ex]
	   \draw[ultra thick, cyan] (4,3)--(1,2)--(2,1);
	   \draw[cyan] (1,3)--(3,1);
	   \draw[thick, cyan] (2,3)--(4,1);
	   \draw[ultra thick, cyan] (1,1)--(4,2)--(3,3);
    	   \gridpoints{3}{4};  
	\end{tikzpicture}
\right\}\right)$ 
 \end{example}

In Section \ref{sec:degeninfinite}, we comment on reticulations that have an empty component (the degenerate case) and reticulations that have an  infinite component (the infinite case) before removing them from further consideration.
In Section \ref{sec:combreg}, we show that finite nondegenerate reticulations have a regular combinatorial structure.   
We discuss reticulations with no repeated families of lines in Section \ref{sec:repfree}.

Next, we consider various means of representing a reticulation.  
In Section \ref{sec:gridembed}, we show how to place the points on a rectangular grid and encode the reticulation as an array (analogous to certain orthogonal arrays) and as a collection of matrices (analogous to sets of mutually orthogonal Latin squares).   
We refer to such arrays and collections of matrices as svelte semi-orthogonal arrays and cooperative systems, respectively.  
We characterize svelte semi-orthogonal arrays and cooperative systems in Sections \ref{sec:svelteSOA} and \ref{sec:halflatinmatrices}. 
The defining conditions for a cooperative system involve column-Latin and row-Latin matrices which act on one another in a manner which 
encodes bireversible automata (we briefly discuss this connection in Section \ref{sec:halflatinmatrices}).
We describe the  geometric structure dual to a reticulation in Section \ref{sec:semitransversaldesign}, which we call  a semitransversal design.
These results parallel the relationship between $(k,n)$-nets, orthogonal arrays,  sets of mutually orthogonal Latin squares, and transversal designs.   
In fact, these known combinatorial objects give rise to examples of reticulations, svelte semi-orthogonal arrays,  cooperative systems, and semitransversal designs. 

In Sections \ref{sec:parastrophy} and \ref{sec:isotopy}, we introduce the notions of parastrophy and isotopy for svelte semi-orthogonal arrays and  cooperative systems. 
As in the theory of Latin squares, these notions have geometric interpretations.

We next turn to  constructions. In Section \ref{sec:smallparam}, we describe reticulations with small lines (one or two points) or one family of either type.    
Three constructions of  reticulations from smaller ones--prolongation, splicing, and a direct product--are described in Sections \ref{sec:prolongation}, \ref{sec:splice}, and \ref{sec:blockproduct}, respectively.

\section{Two extremes}
\label{sec:degeninfinite}

We briefly discuss two extreme cases of reticulations before putting them aside.

\begin{definition}
\label{def:degeneratereticulation}
A reticulation with an empty point set or an empty family of lines is said to be \emph{degenerate}.
A reticulation that is not degenerate is \emph{nondegenerate}.
\end{definition}

\begin{lemma}
\label{lem:degenretic}
Let $\mathcal{P}$ be a set. For any (possibly empty) (multi)set $\mathcal{F}$ of  partitions of $\mathcal{P}$, 
$(\mathcal{P}, \emptyset, \mathcal{F})$ is a degenerate reticulation.  
\end{lemma}

\begin{proof}  
Conditions (R-1) and (R-2) are  satisfied vacuously in this case.
\end{proof}

In Lemma \ref{lem:degenretic}, note that 
$|\mathcal{F}|\leq B_{|\mathcal{P}|}$, the Bell number that counts the number of partitions of $\mathcal{P}$.
The lines of a degenerate reticulation  need not be uniform in size.
We shall see in the next section that the size is uniform in the nondegenerate case.

\begin{example}
In this example, the second, third and fourth families there are lines of size 1 and size 2.
There are five families of lines on three points since $B_3=5$.

\centering
$\left(\{\ \},
\
\left\{ 
\begin{tikzpicture}[scale=.5,baseline=6ex]
\draw[ blue] (1,1)--(1, 3); 
  \gridpoints{3}{1};  
\end{tikzpicture}
,\, 
\begin{tikzpicture}[scale=.5,baseline=6ex]
\draw[ blue] (1,1)--(1, 2); 
\draw[ blue] (.7,3)--(1.3, 3); 
  \gridpoints{3}{1};  
\end{tikzpicture}
,\,
\begin{tikzpicture}[scale=.5,baseline=6ex]
\draw[ blue] (1,2)--(1, 3); 
\draw[ blue] (.7,1)--(1.3, 1); 
  \gridpoints{3}{1};  
\end{tikzpicture}
,\, 
\begin{tikzpicture}[scale=.5,baseline=6ex]
\draw[ blue] (1,1) --(1.5,2)--(1, 3); 
\draw[ blue] (.7,2)--(1.3, 2); 
  \gridpoints{3}{1};  
\end{tikzpicture}
,\,
\begin{tikzpicture}[scale=.5,baseline=6ex]
\draw[ blue] (.7,1)--(1.3, 1); 
\draw[ blue] (.7,2)--(1.3, 2); 
\draw[ blue] (.7,3)--(1.3, 3); 
  \gridpoints{3}{1};  
\end{tikzpicture}
\right\}\right)$.
\end{example}

\begin{definition}
A reticulation is said to be  \emph{finite} whenever its point set and both (multi)sets of families of lines are finite. 
\end{definition}

We offer a simple reticulation with infinite components.

\begin{example}
\label{ex:infinitelines}
Let $\mathcal{P}=\mathbb{R}^2$.  
For each $m\in \mathcal{R}$ and each $b\in \mathcal{R}$,
let $L_{m,b}$ be the line $y=mx+b$, and let $L_{\infty,b}$ be the vertical line $x=b$.  
Let $\mathcal{L}_m=\{L_{m,b} :  b\in \mathbb{R}\}$.   
Let $M$ and $N$ be disjoint subsets of $\mathcal{R}\cup\{\infty\}$.
Let $\mathcal{F}_{\mathrm{weft}}=\{\mathcal{L}_m :  m\in M\}$ and 
       $\mathcal{F}_{\mathrm{warp}}=\{\mathcal{L}_n :  n\in N\}$. 
Any two nonparallel lines meet in a unique point, and every parallel class partitions the plane.
Thus, $(\mathcal{P}, \mathcal{F}_{\mathrm{weft}}, \mathcal{F}_{\mathrm{warp}})$ is a reticulation.

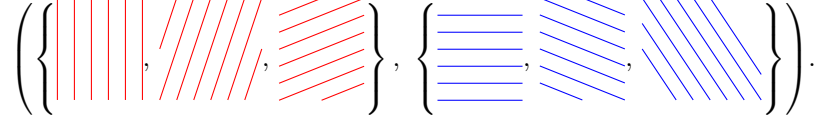
\centering
$\left(\left\{
\begin{tikzpicture}[scale=.45,baseline=3ex]
\foreach \x in {1,1.5,2,2.5,3, 3.5}
  \draw[ red] (\x,0)--(\x,3); 
\end{tikzpicture},\,
\begin{tikzpicture}[scale=.45,baseline=3ex]
\foreach \x in {1,1.5,2,2.5,3}
  \draw[ red] (\x,0)--(\x+1,3); 
\draw[red] (1, 1.5)--(1.5 ,3);
 \draw[ red] (3.5,0)--(4,1.5);
\end{tikzpicture},\,
\begin{tikzpicture}[scale=.45,baseline=6ex]
\foreach \y in {1,1.5,2,2.5,3}
   \draw[red] (0,\y)--(2.5, \y+1); 
\draw[red] (0,3.5)--(1.25, 4);
\draw[red] (1.25,1)--(2.5, 1.5);
\end{tikzpicture}\right\},
\  
\left\{
\begin{tikzpicture}[scale=.45,baseline=6ex]
\foreach \y in {1,1.5,2,2.5,3,3.5,4}
   \draw[ blue] (0,\y)--(2.5, \y); 

\end{tikzpicture},\,
\begin{tikzpicture}[scale=.45,baseline=3ex]
\foreach \y in {1,1.5,2,2.5,3}
   \draw[ blue] (0,\y)--(2.5, \y-1); 
\draw[blue] (1.25,3)--(2.5, 2.5);
\draw[blue] (0,.5)--(1.25, 0);
\end{tikzpicture},\,
\begin{tikzpicture}[scale=.45,baseline=3ex]
\foreach \x in {1,1.5,2,2.5}
   \draw[blue] (\x,0)--(\x-2,3); 
\draw[blue] (-1, 2.25)--(.5 ,0);
\draw[blue] (-1, 1.5)--(0 ,0);
 \draw[blue] (1,3)--(2.5,.75);
\end{tikzpicture}\right\}\right)$.
\end{example}

Example \ref{ex:infinitelines} suggests  the following interpretation of a reticulation:
Taking one family of lines of each type gives a coordinate system for the set of points. 
We develop this further in Section \ref{sec:gridembed}.

Three-webs provide a number of infinite examples as well.  
Three-webs were introduced and studied by Blaschke and his school (for example \cite{MR10451, MR1545014}) in the 1920s and 1930s  
to describe local invariants in  differential geometry arising from configurations of three foliations of the plane  \cite{AKIVIS20001}.     
We can rephrase the  definition from \cite[Sec.~4]{MR35} in our terminology as follows.

\begin{definition} \cite[Sec.~4]{MR35}
A three-web consists of a set of points and three families (types) of lines  which satisfy axioms (R-1) and (R-2). 
\end{definition}

We note that  $(k,n)$-nets (which we recall in Section \ref{sec:combreg}) are finite analogs of three-webs \cite{MR39678}.  
Reticulations are a different sort of analog shaped by their connection to finite bireversible Mealy automata.

Henceforth, we shall consider only finite nondegenerate reticulations.

 \section{Combinatorial regularity}
 \label{sec:combreg}
 
We show that finite nondegenerate reticulations have a regular combinatorial structure.  
The axioms, despite their simplicity, are  generous in this regard.

\begin{theorem}\label{thm:regular}
Let $\mathcal{R}=(\mathcal{P}, \mathcal{F}_{\mathrm{weft}}, \mathcal{F}_{\mathrm{warp}})$ be a finite nondegenerate reticulation.
Let $k=|\mathcal{F}_{\mathrm{weft}}|$ and $\ell=|\mathcal{F}_{\mathrm{warp}}|$.
Suppose that there is a weft line of size $m$ and a warp line of size $n$.  Then the following hold.
\begin{enumerate}
\item   There are $(mn)$-many points.
\item   Every point is in $k$-many weft lines (one from each weft family).
\item   Every point is in $\ell$-many warp lines (one from each warp family).
\item   Every weft line contains $m$-many points.
\item   Every warp line contains $n$-many points.
\item   Each family of weft lines contains $n$-many  parallel  lines.
\item   Each family of warp lines contains $m$-many parallel  lines.
\end{enumerate}
We call the tuple $(m,n,k,\ell)$ the \emph{parameters} of $\mathcal{R}$,
we say that $\mathcal{R}$ is an $(m,n,k,\ell)$-reticulation,
and we say that $\mathcal{R}$ has \emph{order} $m\times n$.
\end{theorem}

\begin{proof} 
Each family of  lines in $\mathcal{F}_{\mathrm{weft}}$ partitions the set of points by (R-2), so
every point is in exactly one line from each of the $k$-many families of weft lines.  
 Similarly, every point is in $\ell$-many warp lines. Thus, (ii) and (iii) hold.
In any family of weft lines, there is one line for each of the $n$-many points in the given warp line since each weft line intersects the given warp line in exactly one point by (R-1). 
Thus, there are $n$-many lines in each family of weft-lines, giving (vi).  A similar argument gives (vii).  
The total number of weft lines (with multiplicity) is counted by the size $n$ of the given warp line multiplied by the number $k$ of weft lines containing each of the points on the warp line by (iii).   
Making the same count of weft lines starting from a warp line of size $w$ gives $nw$-many weft lines.  Now $nk=nw$ so $w=k$.  
Thus, every warp line has size $n$, as claimed in (v).  Reversing the roles of weft and warp gives (iv).
For (i), count the number of points by multiplying the number $n$ of weft lines multiplied by the size $m$ of each weft line in each parallel family since the weft lines partition the point set (R-2). 
\end{proof} 

\begin{example}
The reticulations in Examples  \ref{ex:grid} and \ref{ex:smallrr}
have respective parameters (3,4,1,1) and (3,4,2,2).
\end{example}

\begin{lemma}
\label{lem:transposereticulation}
If $\mathcal{R}=(\mathcal{P}, \mathcal{F}_{\mathrm{weft}}, \mathcal{F}_{\mathrm{warp}})$ is an $(m,n,k,\ell)$-reticulation,  
then the triple $\mathcal{R}^\dag=(\mathcal{P}, \mathcal{F}_{\mathrm{warp}},  \mathcal{F}_{\mathrm{weft}})$ is an $(n,m,\ell,k)$-reticulation. 
We refer to $\mathcal{R}^\dag$ as the \em{transpose} of $\mathcal{R}$.
\end{lemma}

We note the following.

\begin{lemma}
Let $(\mathcal{P}, \mathcal{F}_{\mathrm{weft}}, \mathcal{F}_{\mathrm{warp}})$ be a finite reticulation.
\begin{enumerate}
\item Any two points are incident with at most one line in each family of lines.
\item Given a point and a line not incident to the point, 
         there is one line in each family of the other type of line incident with the given point and meeting the given line.
\end{enumerate}
\end{lemma}

\begin{proof} 
A point  is in exactly one line in any family of parallel weft/warp lines.  
A second point is either on this line or not, giving (i).
The line containing the given point meets any given warp/weft line in a unique point.  Thus, (ii) holds.
\end{proof} 

We now compare reticulations to geometric structures known as nets.  

\begin{definition}
An \emph{$(r,k)$-net} or \emph{Bruck net of order $r$ and degree $k$}   is an incidence structure consisting of $k^2$-many points and $rk$-many lines such that 
the lines comprise $r$-many parallel classes, each containing $k$-many lines (so each parallel class partitions the set of points), 
every line contains $k$-many points, and every point is incident with $r$-many lines, one from each parallel class.
\end{definition}

We note that $(r,k)$-nets are a generalization of affine geometries and a special case of partial geometries  \cite{MR1474497,MR782310,MR1434062,MR804798}. 
We immediately observe the following.

\begin{proposition}
\label{prop:net->retic}
Any partition of the lines of an $(m, k+\ell)$-net into parts of size  $k$ and $\ell$
is an $(m, m, k, \ell)$-reticulation.
\end{proposition}

It is known that an $(r,k)$-net is equivalent to a  set of $(k-2)$-many mutually orthogonal Latin squares of order $r$ 
\cite[Section 11.1]{MR1096296}, \cite[Subsection 1.4.1]{MR3837138}, \cite[Section 15.3]{MR1644242}, \cite[Section 8.2]{MR3495977}.
We shall develop several results analogous to known results concerning Latin squares, such as parastrophy and isotopy.

\section{Repetition-free reticulations}
\label{sec:repfree}

Reticulations  with parameters $(m,n,2,2)$ arose in connection with bireversible finite automata \cite{CurtinSavchuk:combBRA}.  
The reticulations arising from this application may have repeated families of lines.  Allowing this sort of repetition also side-steps some technical issues around preventing it.  
Furthermore, repetition of this nature arises in connection with orthogonal arrays, an analog of which we connect to reticulations in Section \ref{sec:svelteSOA}.  
The axioms of a reticulation allow multisets of families of lines and have no interaction between families of the same type. This leads to the following.

\begin{proposition}
\label{prop:changemultiplicities}
Given a reticulation, freely changing the multiplicities of the families of parallel lines to any nonnegative value yields another reticulation.
 \end{proposition}

\begin{definition}
A  reticulation is said to be \emph{repetition-free} whenever each line family appears with multiplicity one. 
The  \emph{foundation} of a reticulation $(\mathcal{P},\mathcal{F}_{\mathrm{weft}},\mathcal{F}_{\mathrm{warp}})$ is the repetition-free  reticulation $(\mathcal{P},\tilde{\mathcal{F}}_{\mathrm{weft}},\tilde{\mathcal{F}}_{\mathrm{warp}})$, 
where $\tilde{\mathcal{F}}_{\mathrm{weft}}$ and $\tilde{\mathcal{F}}_{\mathrm{warp}}$ are the underlying sets on $\mathcal{F}_{\mathrm{weft}}$ and $\mathcal{F}_{\mathrm{warp}}$, respectively.
\end{definition}

For the purposes of construction and counting, one may focus on repetition-free reticulations.  
If this were one's focus, one could choose to make repetition-free a part of the definition.

\begin{example}
\label{ex:repeatedgrid}
 The reticulation in Example \ref{ex:grid} is repetition-free.  It is the foundation of the following reticulation formed by repeating families of lines.
 
 \centering
$\left(\left\
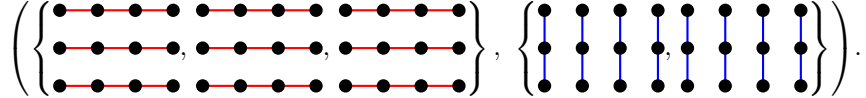
{
 \begin{tikzpicture}[scale=.5,baseline=6ex]
\foreach \y in {1,2, 3}
{ \draw[thick, red] (1,\y)--(4,\y); };
\gridpoints{3}{4}
\end{tikzpicture},\,
 \begin{tikzpicture}[scale=.5,baseline=6ex]
\foreach \y in {1,2, 3}
{ \draw[thick, red] (1,\y)--(4,\y); };
\gridpoints{3}{4}
\end{tikzpicture},\,
 \begin{tikzpicture}[scale=.5,baseline=6ex]
\foreach \y in {1,2, 3}
{ \draw[thick, red] (1,\y)--(4,\y); };
\gridpoints{3}{4}
\end{tikzpicture}\right\},
\ 
\left\{
 \begin{tikzpicture}[scale=.5,baseline=6ex]
\foreach \x in {1,2,3, 4}
{ \draw[thick, blue] (\x,1)--(\x, 3); };
\gridpoints{3}{4}
\end{tikzpicture},\,
 \begin{tikzpicture}[scale=.5,baseline=6ex]
\foreach \x in {1,2,3, 4}
{ \draw[thick, blue] (\x,1)--(\x, 3); };
\gridpoints{3}{4}
\end{tikzpicture}\right\}\right)$.
\end{example}

\begin{lemma}
Given a repetition-free $(m,n,k,\ell)$-reticulation $\mathcal{R}$, 
the number of $(m,n,k+w, \ell+x)$-reticulations having $\mathcal{R}$ as its foundation is
\[\binom{k+w-1}{w}\binom{\ell+x-1}{x}.\]
\end{lemma}

\begin{proof} 
Form a reticulation with the desired parameters having $\mathcal{R}$ as its foundation by 
choosing $w$ of the $k$-many weft families with repetition and $x$ of the $\ell$-many warp families from $\mathcal{R}$.  
\end{proof}

\begin{lemma}
Let $\nu(m,n,k,\ell)$ and $\lambda(m,n,k,\ell)$ respectively denote the number of reticulations and repetition-free $(m,n,k,\ell)$-reticulations.
Then 
\[ 
\nu(m,n,k,\ell)= \sum_{w=0}^{k-1} \sum_{x=0}^{\ell-1} \binom{k-1}{w}\binom{\ell-1}{x}  \lambda(m,n, k-w, \ell-x).
\]
\end{lemma}

\begin{definition}
We say that a reticulation 
      $(\mathcal{P},\mathcal{F}_{\mathrm{weft}},\mathcal{F}_{\mathrm{warp}})$ 
\emph{contains} another  reticulation 
     $(\tilde{\mathcal{P}},\tilde{\mathcal{F}}_{\mathrm{weft}},\tilde{\mathcal{F}}_{\mathrm{warp}})$
whenever $\tilde{\mathcal{P}}=\mathcal{P}$,     
                $\tilde{\mathcal{F}}_{\mathrm{weft}}\subseteq \mathcal{F}_{\mathrm{weft}}$, and
                $\tilde{\mathcal{F}}_{\mathrm{warp}}\subseteq \mathcal{F}_{\mathrm{warp}}$.
\end{definition}

\begin{lemma}
A repetition-free $(m,n,k,\ell)$-reticulation contains $\binom{k}{w}\binom{\ell}{x}$-many distinct repetition-free $(m,n, w,x)$-reticulations.
\end{lemma}

\begin{proof} 
Given a repetition-free $(m,n,k,\ell)$-reticulation, any subset of weft lines of size $w$  with any subset of warp lines of size $x$ gives a repetition-free $(m,n, w,x)$-reticulation.  
There are $\binom{k}{w}\binom{\ell}{x}$-many such reticulations, all distinct since the original reticulation is repetition-free.
\end{proof}

As a followup to Proposition \ref{prop:changemultiplicities}, we note the following.

\begin{lemma}
\label{lem:reticbypair}
Let $\mathcal{P}$ be a nonempty set, and let $\mathcal{F}_{\mathrm{weft}}$ and $\mathcal{F}_{\mathrm{warp}}$ be (multi)sets of subsets of $\mathcal{P}$.  
Then $(\mathcal{P}, \mathcal{F}_{\mathrm{weft}}, \mathcal{F}_{\mathrm{warp}})$ is a reticulation if and only if 
 for every $\mathcal{E}\in  \mathcal{F}_{\mathrm{weft}}$  and every $\mathcal{A}\in  \mathcal{F}_{\mathrm{warp}}$, 
$(\mathcal{P}, \{\mathcal{E}\}, \{\mathcal{A}\})$ is a reticulation.
\end{lemma}

\begin{proof} 
Clear from Definition \ref{def:reticulation}.
\end{proof}

\begin{example}
Distinct families of lines in a reticulation may contain the same line, e.g.,

{
\centering
$\left(\left\
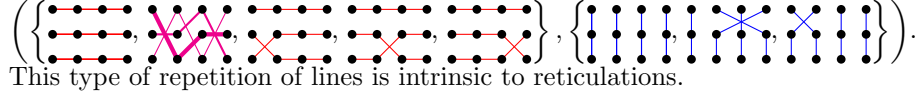
{
\begin{tikzpicture}[scale=.33,baseline=4ex]
\foreach \x in {1,2,3, 4}. \foreach \y in {1,2, 3}
{ \draw[ red] (1,\y)--(4,\y); };
\gridpoints{3}{4}
\end{tikzpicture},\,
\begin{tikzpicture}[scale=.33,baseline=4ex]
\draw[thick,magenta] (1,2)--(2,2) -- (3, 3) -- (4,1) ;
\draw[magenta] (1,1)--(2,3)--(3,1)--(4,3); 
\draw[ultra thick, magenta]  (1,3)--(2,1)--(3, 2)--(4, 2);
\gridpoints{3}{4}
\end{tikzpicture},\,
\begin{tikzpicture}[scale=.33,baseline=4ex]
\draw[ red] (1,3)--(4,3);
\draw[ red] (1,2)--(2,1)--(3,1)--(4,1);
\draw[ red] (1,1)--(2,2)--(3,2)--(4,2);
\gridpoints{3}{4} 
\end{tikzpicture},\,
\begin{tikzpicture}[scale=.33,baseline=4ex]
\draw[ red] (1,3)--(4,3);
\draw[ red] (1,2)--(2,2)--(3,1)--(4,1);
\draw[ red] (1,1)--(2,1)--(3,2)--(4,2);
\gridpoints{3}{4}
\end{tikzpicture},\,
\begin{tikzpicture}[scale=.33,baseline=4ex]
\draw[ red] (1,3)--(4,3);
\draw[ red] (1,2)--(2,2)--(3,2)--(4,1);
\draw[ red] (1,1)--(2,1)--(3,1)--(4,2);
\gridpoints{3}{4}
\end{tikzpicture}
\right\},
\left\{
 \begin{tikzpicture}[scale=.33,baseline=4ex]
\foreach \x in {1,2,3, 4}. \foreach \y in {1,2, 3}
{ \draw[ blue] (\x,1)--(\x, 3); };
\gridpoints{3}{4}
\end{tikzpicture},\,
 \begin{tikzpicture}[scale=.33,baseline=4ex]
 \draw[ blue] (1,1)--(1, 3); 
 \draw[blue] (2,3)--(4,2)--(4,1);
  \draw[ blue] (3,1)--(3, 3); 
  \draw[blue] (4,3)--(2,2)--(2,1);
\gridpoints{3}{4}
\end{tikzpicture},\,
 \begin{tikzpicture}[scale=.33,baseline=4ex]
  \draw[ blue] (1,1)--(1,2)--(2, 3); 
 \draw[ blue] (2,1)--(2,2)--(1, 3); 
   \draw[ blue] (3,1)--(3, 3); 
     \draw[ blue] (4,1)--(4, 3); 
\gridpoints{3}{4}
\end{tikzpicture}
\right\}\right)$.
}

\noindent
This type of repetition of lines is intrinsic to reticulations.   
\end{example}

Repetition-free reticulations can also be defined in terms of hypergraphs. 
The  \emph{complete $m$-uniform hypergraph} on some set of vertices has edge set consisting
of all subsets of vertices of size $m$.
A \emph{perfect matching} in a hypergraph is a subset of hyperedges that partition the vertices.  
Two hypergraphs on a common vertex set are \emph{orthogonal} when each hyperedge of the first meets each hyperedge of the second in a unique vertex.
Fix a vertex set $\mathcal{P}$ of size $mn$, and 
let $M$ and $N$ be the complete $m$- and $n$-uniform hypergraphs on $\mathcal{P}$.
A repetition-free $(m,n,k,\ell)$-reticulation consists of a set $\mathcal{F}_{\hbox{weft}}$ of $k$-many perfect matchings of $M$ 
and a set $\mathcal{F}_{\hbox{warp}}$ of $\ell$-many perfect matchings of $N$ such that each matching in $\mathcal{F}_{\hbox{weft}}$ is orthogonal to each matching in $\mathcal{F}_{\hbox{warp}}$.  

The combinatorial regularity of finite nondegenerate reticulations gives them a flavor different from typical hypergraph problems. 
Pairs of hypergraphs related by conditions similar to those of reticulations (focusing on covering properties rather than matching properties) have been studied as cross-intersecting hypergraphs \cite{MR3603314} and looms \cite{MR4778802}.   
One may also consider families of sets analogous to cross-intersecting sets \cite{MR1357781}.

We note the following.

 \begin{lemma}
 Fix a finite set of points $\mathcal{P}$ and a set $\mathcal{G}$ of partitions of $\mathcal{P}$.
 Then the set $\mathbb{F}_\mathcal{G}$ of sets of partitions $\mathcal{F}$ of $\mathcal{P}$ such that 
 $(\mathcal{P},\mathcal{F}, \mathcal{G})$ is a repetition-free reticulation make up the independent sets of a matroid.  
 \end{lemma}

 \begin{proof}
By Lemma \ref{lem:degenretic}, $\emptyset\in \mathbb{F}_\mathcal{G}$.  
By Proposition \ref{prop:changemultiplicities}, for each  $\mathcal{F}\in\mathbb{F}_\mathcal{G}$,  any subset of $\mathcal{F}$ is also in $\mathbb{F}_\mathcal{G}$. 
Suppose $\mathcal{F}_1$, $\mathcal{F}_2 \in\mathbb{F}_\mathcal{G}$ with $|\mathcal{F}_1|>|\mathcal{F}_2|$.  
Then there is some $F\in \mathcal{F}_1\setminus \mathcal{F}_2$.
In light of Lemma \ref{lem:reticbypair}, $\mathcal{F}_2\cup\{ F \}\in\mathbb{F}_\mathcal{G}$.
Thus, the lemma holds. 
 \end{proof} 
 
Note that many choices of $\mathcal{G}$ admit no nonempty set partitions $\mathcal{F}$ such that  $(\mathcal{P},\mathcal{F}, \mathcal{G})$ is a reticulation.   
In such a case, $\mathbb{F}_\mathcal{G}=\{\emptyset\}$, giving a degenerate reticulation as in Lemma \ref{lem:degenretic}.
For example, if the partitions in $\mathcal{G}$ do not have the same number of cells (each of the same size), 
then $\mathbb{F}_\mathcal{G} = \{\emptyset\}$ by Theorem \ref{thm:regular}.

%
\section{Representing reticulations}
\label{sec:gridembed}

We describe how to embed a finite nondegenerate reticulation on a grid and then offer two other 
representations involving arrays and matrices that arise from this concrete realization.

\begin{notation}
Let $m$, $n$, $k$, $\ell$ be positive integers.
Write  $[m]=\{1,2, \ldots,m\}$.
Write $[m]^{k}  \times [n]^{\ell}$ to denote the Cartesian product of $k$-many copies of $[m]$ and  $\ell$-many copies of $[n]$.
It will be convenient to write $y\in[m]^{k}  \times [n]^{\ell}$ as $y=(y_1|y_2)$, where $y_1\in [m]^{k}$ (a $k$-tuple of elements of $[m]$) and $y_2\in [n]^{\ell}$ (an $\ell$-tuple of elements of $[n]$).
 Let $ \sym{m}$ denote the symmetric group on $[m]$.
Let $\mathcal{M}_{m\times n}(X)$ denote the set of $m\times n$ 
matrices with entries from the set $X$.   For any matrix $M$, we write $M^\dag$ to denotes its transpose, and we write $M(i,j)$ to denote the $(i,j)$-entry of $M$.
\end{notation}

\begin{definition}\label{def:constructarray}
Let $\mathcal{R}=(\mathcal{P}, \mathcal{F}_{\mathrm{weft}}, \mathcal{F}_{\mathrm{warp}})$ be a 
$(m,n,k,\ell)$-reticulation.
We say that $\mathcal{R}$ is an \emph{ordered reticulation} or \emph{ordered $(m,n,k,\ell)$-reticulation} 
when it is endowed with the following:
\begin{enumerate}[1.]
\item a fixed choice of weft and warp (multi)sets of lines;
\item an enumeration $\mathcal{E}^1$, $\ldots$, $\mathcal{E}^k$ of the $k$-many weft families;
\item an enumeration of the lines in each weft family $\mathcal{E}^u=\{E_i^u\}_{i\in[m]}$ ($u\in[k]$);
\item an enumeration $\mathcal{A}^1$, $\ldots$, $\mathcal{A}^\ell$ of the  $\ell$-many parallel families; and
\item an enumeration of the lines in each warp family $\mathcal{A}^v=\{A_j^v\}_{j\in[n]}$ ($v\in[\ell]$).
\end{enumerate}
\end{definition} 
We now give several representations of an ordered reticulation.

\begin{definition}
\label{def:gridarraymatrices}
Take the ordered $(m,n,k,\ell)$-reticulation of in Definition \ref{def:constructarray}:\\
 $\mathcal{R}=(\mathcal{P}, 
                              \mathcal{F}_{\mathrm{weft}}=\{\mathcal{E}^u=\{E_i^u\}_{i\in[m]}\}_{u\in[k]},
                               \mathcal{F}_{\mathrm{warp}}=\{\mathcal{A}^v=\{A_j^v\}_{j\in[n]}\}_{v\in[\ell]})$.

\begin{enumerate}[1.]  \setcounter{enumi}{5}
\item Let $P_{i,j}$ be the unique point of intersection of $E_i^1$ and $A_j^1$ $(i\in[m], j\in[n])$.
\item For $u\in[k]$, define $C^u\in\mathcal{M}_{m\times n}([m])$ by $C^u(i,j) = q$ when $P_{i,j}\in E^u_q$.
\item For $v\in[\ell]$, define  $R^v\in\mathcal{M}_{m\times n}([n])$ by $R^v(i,j) = r$ when $P_{i,j}\in A^v_r$.
\item Set $\sigma(\mathcal{R})=\{ (C^1(i,j), \ldots, C^k(i,j); R^1(i,j), \ldots, R^\ell(i,j) ): i\in[m], j\in[n]\}$.
                      We refer to $\sigma(\mathcal{R})$ as the \emph{array of $\mathcal{R}$}.

\item Write $\mu(\mathcal{E}^u)=C^u$, $\mu(\mathcal{A}^v)=R^v$, and
                  $\mu(\mathcal{R})=(C^1,C^2, \ldots, C^k; R^1, R^2, \ldots, R^\ell)$.
                       We refer to $\mu(\mathcal{R})$ as the \emph{cooperative matrices of $\mathcal{R}$}. 
\end{enumerate}
\end{definition}

A finite reticulation  can be ordered in many ways.   We discuss how the  arrays and matrices constructed from different orderings in Section \ref{sec:parastrophy} and \ref{sec:isotopy}.

\begin{example}
\label{ex:gridarraymatrix}
Name the points in  Example \ref{ex:smallrr} as $\mathcal{P}=\left\{\addtolength{\arraycolsep}{-0.4em}
\begin{array}{cccc}
a&b&c&d\\[-2pt]
e&f&g&h\\[-2pt]
i&j&k&\ell
\end{array}\right\}.$
\begin{itemize}
\item[2, 3.] Enumerate weft families/lines:\\
$\begin{array}{llll}
\mathcal{E}^1=&\{ E_1^1=\{a,j,g,h\}, &E_2^1=\{ e,f, c,\ell\}, &E_3^1=\{i,b,k,d\} \},\\
\mathcal{E}^2=&\{  E_1^2=\{e,f,g,h\}, &E_2^2=\{i,j,k,\ell\}, &E_3^2=\{a,b,c,d\} \}.
\end{array}$

\item[4, 5.] Enumerate warp families/lines:\\
$\begin{array}{lllll}
\mathcal{A}^1=\!\!&\!\!\{ A_1^1=\{i,c,h\}, \!\!&A_2^1=\{ j,e,d\}, &\!\!A_3^1=\{a, f,k\}, &\!\!A_4^1=\{b,g,\ell\} \},\\
\mathcal{A}^2=\!\!&\!\!\{ A_1^2=\{ a,e,i\}, &\!\!A_2^2=\{ b,f,j\}, &\!\!A_3^2=\{c,g,k\}, &\!\!A_4^2= \{d,h,\ell\}\}.
\end{array}$

\item[6.] Name points according to the grid given by the first line of each type:\\ 
 $\begin{array}{cccc}
P_{1,1} = h, & P_{1,2} = j, & P_{1,3}=a, & P_{1,4}=g,\\
P_{2,1} = c, &P_{2,2} = e, & P_{2,3}=f, & P_{2,4}=\ell,\\
P_{3,1} = i, &P_{3,2} = d, & P_{3,3}=k, & P_{3,4}=b.
\end{array}$\\
Note that the  first family of lines of each type form a grid, as in Example \ref{ex:grid}.

\item[7, 8, 10.] Construct $C$'s and $R$'s

$\begin{array}{cccc}
 h\in E_{\mathbf 1}^1 &  j \in E_{\mathbf 1}^1&a \in E_{\mathbf 1}^1&g\in E_{\mathbf 1}^1\\
 c\in E_{\mathbf 2}^1& e\in E_{\mathbf 2}^1  & f\in E_{\mathbf 2}^1 & \ell\in E_{\mathbf 2}^1\\
 i\in E_{\mathbf 3}^1& d\in E_{\mathbf 3}^1 & k\in E_{\mathbf 3}^1 & b\in E_{\mathbf 3}^1
\end{array}$, 
so 
$C^1 =\left[\addtolength{\arraycolsep}{-0.4em}\begin{array}{cccc}
1&1&1&1\\[-2pt]
2&2&2&2\\[-2pt]
3&3&3&3
\end{array}\right]$.

$\begin{array}{cccc}
h\in E_{\mathbf 1}^2 &  j \in E_{\mathbf 2}^2& a \in E_{\mathbf 3}^2& g\in E_{\mathbf 1}^2\\
c\in E_{\mathbf 3}^2&e\in E_{\mathbf 1}^2   & f\in E_{\mathbf 1}^2 & \ell\in E_{\mathbf 2}^2\\
 i\in E_{\mathbf 2}^2& d\in E_{\mathbf 3}^2  & k\in E_{\mathbf 2}^2 & b\in E_{\mathbf 3}^2
\end{array}$,
so 
$C^2 =\left[\addtolength{\arraycolsep}{-0.4em}\begin{array}{cccc}
1&2&3&1\\[-2pt]
3&1&1&2\\[-2pt]
2&3&2&3
\end{array}\right]$.

$\begin{array}{cccc}
h\in A_{\mathbf 1}^1 &  j \in A_{\mathbf 2}^1& a \in A_{\mathbf 3}^1& g\in A_{\mathbf 4}^1\\
c\in A_{\mathbf 1}^1& e\in A_{\mathbf 2}^1  & f\in A_{\mathbf 3}^1 & \ell\in A_{\mathbf 4}^1\\
i\in A_{\mathbf 1}^1& d\in A_{\mathbf 2}^1  & k\in A_{\mathbf 3}^1 & b\in A_{\mathbf 4}^1
\end{array}$, 
so 
$R^1 =\left[\addtolength{\arraycolsep}{-0.4em}\begin{array}{cccc}
1&2&3&4\\[-2pt]
1&2&3&4\\[-2pt]
1&2&3&4
\end{array}\right]$.

$\begin{array}{cccc}
h\in A_{\mathbf 4}^2 &  j \in A_{\mathbf 2}^2& a \in A_{\mathbf 1}^2& g\in A_{\mathbf 3}^2\\
 c\in A_{\mathbf 3}^2& e\in A_{\mathbf 1}^2  & f\in A_{\mathbf 2}^2 & \ell\in A_{\mathbf 4}^2\\
i\in A_{\mathbf 1}^2& d\in A_{\mathbf 4}^2 & k\in A_{\mathbf 3}^2 & b\in A_{\mathbf 2}^2
\end{array}$, 
so 
$R^2 =\left[\addtolength{\arraycolsep}{-0.4em}\begin{array}{cccc}
4&2&1&3\\[-2pt]
3&1&2&4\\[-2pt]
1&4&3&2
\end{array}\right]$.

\item[9.] Construct $\sigma(\mathcal{R})$\\
$\sigma(\mathcal{R})=\{ (1,1;1,4), (1,2;2,2), (1,3;3,1), (1,1;4,3)$,\\
      $\phantom{X=\{111}(2,3;1,3), (2,1;2,1), (2,1;3,2), (2,2;4,4)$,\\
       $\phantom{X=\{111}(3,2;1,1), (3,3;2,4), (3,2;3,3), (3,3;4,2)\}$.
\end{itemize}
\end{example}

The following justifies our presentation of reticulations on grids  and will be helpful in  further discussions and constructions of reticulations. 

\begin{lemma}
\label{lem:gridembed}
With reference to Definitions \ref{def:constructarray} and \ref{def:gridarraymatrices},
define $\theta:\mathcal{P}\rightarrow [m]\times [n]$  by 
\[ 
       \theta(P_{i,j})= (i,j),
 \] 
and map $\theta$ over lines  and sets of lines:
\begin{align}
   \theta(E^u_q) &= \{\theta(P) :  P\in E^u_q\}, & \theta(\mathcal{E}^u ) &= \{\theta(E^u_q)  :  q\in[m] \},\nonumber\\
   \theta({A}^v_r) &= \{\theta(P) :  P\in A^v_r\},  & \theta(\mathcal{A}^v ) &= \{\theta(A^v_r)  :  r\in[n] \},\nonumber
\end{align}
\[
   \theta(\mathcal{R}) = ( \theta(\mathcal{E}^1), \ldots, \theta(\mathcal{E}^k); 
                                                                            \theta(\mathcal{A}^1), \ldots, \theta(\mathcal{A}^\ell) ).
\]

\begin{enumerate}
\item The map  $\theta$ is a bijection that induces an embedding of $\mathcal{R}$ onto the grid $[m]\times [n]$. 
\item  $C^1(i,j)=i$.  
         So $\theta$ maps $E^1_i$ to the horizontal line $\{ (i,1), \ldots, (i, n)\}$.
\item $R^1(i,j)=j$.
        So $\theta$ maps $A^1_j$ to the vertical line $\{(1,j), \ldots, (m, j)\}$.
\item $C^u(i,j)=q$, where $(i,j)$ is on the line $\theta(E^u_q)$.
           In particular, $\theta$ maps $E^u_q$ to the line consisting of the positions where $C^u$ has entry $q$.
\item $R^v(i,j)=r$, where $(i,j)$ is on the line $\theta(A^v_r)$.
            In particular, $\theta$ maps $A^v_r$ to the line consisting of the positions where $R^v$ has entry $r$.
\end{enumerate}
\end{lemma}

\begin{proof} 
Parts (i)--(iii) follow from the fact that there are $n$-many weft lines of size $m$ in $\mathcal{E}^1$ and $m$-many warp lines of size $n$ in $\mathcal{A}^1$, and 
      each line of one type meets each line of the other type in exactly one point.  
The indices on $P_{i,j}$ are defined by its incidence with line $i$ of $\mathcal{E}^1$ and line $j$ of $\mathcal{A}^1$.  Parts (iv) and (v) follow from  the definitions of $C^u$ and $R^v$.
\end{proof} 

Lemma \ref{lem:gridembed} gives the following deconstruction the construction. 
Given a reticulation, a chosen weft family and a chosen warp family are sorted into horizontal and vertical lines, respectively,  to arrange the point set on $[m]\times [n]$. 
From each  $C^u$ and $R^v$ we read the corresponding weft family and warp family, respectively,
   by the rule that the positions of the cooperative matrices which hold the same value are on the same line of the corresponding family.  
We examine the conditions on arrays that permit this construction of a reticulation in the next section.

\begin{example}
\label{ex:newgridembedding}
The embedding of the reticulation found in Example \ref{ex:gridarraymatrix} is 

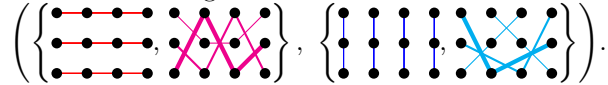
\centering
$\left(\left\{\begin{tikzpicture}[scale=.4,baseline=4.7ex]
\foreach \x in {1,2,3, 4}. \foreach \y in {1,2, 3}
{ \draw[ red] (1,\y)--(4,\y); };
\gridpoints{3}{4}
\end{tikzpicture},\, 
\begin{tikzpicture}[scale=.4,baseline=4.7ex]
\draw[magenta] (1,3)--(2,2)--(3, 2 )--(4,3);
\draw[ultra thick, magenta] (1,1)--(2,3)--(3,1)--(4,2);
\draw[thick, magenta] (1,2)--(2,1)--(3,3)--(4,1);
\foreach \x in {1,2,3, 4} \foreach \y in {1,2, 3}
   \fill[black] (\x,\y) circle(5pt) ; 
\end{tikzpicture}\right\},
\ 
\left\{ \begin{tikzpicture}[scale=.4,baseline=4.7ex]
\foreach \x in {1,2,3, 4}. \foreach \y in {1,2, 3}
{ \draw[ blue] (\x,1)--(\x, 3); };
\gridpoints{3}{4}
\end{tikzpicture}
,\,
\begin{tikzpicture}[scale=.4,baseline=4.7ex]
\draw[ cyan] (1,1)--(2,2)--(3,3);
\draw[cyan] (2,3)--(3,2)--(4,1);
\draw[thick, cyan] (1,2)--(3,1)--(4,3);
\draw[ultra thick, cyan] (1,3)--(2,1)--(4,2);
\gridpoints{3}{4}
\end{tikzpicture}
\right\}\right)$.
 \end{example}

Different choices of orderings in Definition \ref{def:constructarray} may lead to 
different embeddings onto the grid.  Indeed, in Example \ref{ex:gridarraymatrix} we started 
with one embedding onto the grid and ended with the embedding shown in Example \ref{ex:newgridembedding}.
One may omit $C^1$ and $R^1$ when listing the cooperative matrices as they simply establish the grid lines and have fixed values.
Had we used the natural order from the grid embedding drawn in Example \ref{ex:smallrr}, we could have immediately read that 
\[ \addtolength{\arraycolsep}{-0.4em}
C^2=\begin{bmatrix} 1&3&2&3\\[-2pt] 2&2&1&1\\[-2pt] 3&1&3&2 \end{bmatrix}
\qquad\hbox{and}\qquad
R^2=\begin{bmatrix} 1&2&3&4\\[-2pt] 4&1&2&3\\[-2pt] 3&4&1&2 \end{bmatrix}.\]
We elaborate on this observation in Section \ref{sec:parastrophy}.

For the purposes of bookkeeping, one may organize components of a reticulation.  
We will discuss the impact of changing the orderings in Sections \ref{sec:parastrophy} and \ref{sec:isotopy}.

\begin{definition}
Given an $(m,n,k,\ell)$-reticulation, an \emph{initial ordering} consists of a choice of weft and warp families $\mathcal{E}^1$ and $\mathcal{A}^1$ and 
orderings of the lines in these families $\mathcal{E}^1=\{E^1_1, \ldots, E^1_m\}$ and  $\mathcal{A}^1= \{A^1_1, \ldots, A^1_n\}$. 
The \emph{standard ordering relative to the initial ordering} on the lines in the other $\mathcal{E}^u$ and $\mathcal{A}^v$ (whose order is yet to be fixed) is given by the rule 
that $E^u_i$ is the line containing $P_{i,1}$ $(i\in[m])$ and $A^v_j$ is the line containing $P_{1,j}$ $(j\in[n])$.  
The \emph{standard orderings of the $\mathcal{E}^u$ and $\mathcal{A}^v$ relative to the initial ordering} is the ordering 
     for which the corresponding $C^u$ are in row-lexicographic order and the corresponding $R^v$ are in  column-lexicographic order.
\end{definition}

Note that an initial ordering can be made in $k\cdot \ell\cdot m! \cdot n!$ ways. 
Relative to a standard ordering, the first column of all $C^u$ is $[1 \, 2 \, \ldots \, m]^\dag$ and the first row of all $A^v$ is $[1\, 2\, \ldots \, n]$. 
In  row-lexicographic order  $C^1=H$ is always first, and in column-lexicographic order  $R^1=V$ is always first.  One may also swap the roles of weft and warp to make $m\leq n$.

\section{Svelte semi-orthogonal arrays}
\label{sec:svelteSOA}

We examine the array $\sigma(\mathcal{R})$ of an ordered reticulation $\mathcal{R}$.  
We introduce a combinatorial structure (svelte semi-orthogonal arrays) which  characterize the arrays of ordered reticulations.  
This is analogous to the relationship between $(r,k)$-nets and certain orthogonal arrays (see \cite[Section 11.1]{MR3495977}).

\begin{definition}
By a \emph{svelte  semi-orthogonal array with parameters $(m,n,k,\ell)$}, we mean a subset 
$\mathcal{Y}$ of $[m]^{k}  \times [n]^{\ell}$ with the property that for all $u\in [k]$ and $v\in [\ell]$, 
every element of $[m]\times[n]$ appears exactly once among the pairs $(y_1(u),y_2(v))$ as $(y_1|y_2)$ runs over  $\mathcal{Y}$.
\end{definition}

A svelte semi-orthogonal array with parameters $(m,n,k,\ell)$ contains $(mn)$-many $(k+\ell)$-tuples.
We introduced the special case with parameters $(m,n,2,2)$ in \cite{CurtinSavchuk:combBRA} in connection with bireversible finite automata (with components listed in a different order).  
We called such an instance a \emph{narrow semi-orthogonal array}.  
The adjectives ``svelte'' and ``narrow'' convey the understanding that a more general notion of semi-orthogonal array could be formulated and studied; 
however, doing so lies beyond the scope of the present work.

\begin{theorem}
With reference to Definition \ref{def:constructarray}, 
$\sigma(\mathcal{R})$ is a svelte semi-orthogonal array with parameters $(m,n,k,\ell)$.
\end{theorem}

\begin{proof} 
By  construction, $\sigma(\mathcal{R})\subset [m]^{k}  \times [n]^{\ell}$.  
The construction of $\sigma(\mathcal{R})$ gives that 
if the lines $E^u_q$ and $A^v_r$ both contain the point $P_{i,j}$, then $(i=y_1(1), y_1(u), j=y_2(1),y_2(v))=(i,q,j,r)$. 
Every point is in one weft line from family $u$ and one warp line from family $v$, so as $P_{i,j}$ runs over the points
(that is, as $(i,j)$ runs over $[m]\times [n]$) the pair $(y_1(u),y_2(v))=(q,r)$ also runs over $[m]\times[n]$.
\end{proof} 

\begin{theorem}
\label{thm:TSOA-R}
Let $\mathcal{Y}$ be a svelte semi-orthogonal array with parameters $(m,n,k,\ell)$.
\begin{enumerate}
\item For each $u\in [k]$, $v\in[\ell]$ and for each $(q,r)\in[m]\times[n]$, 
         there is a unique $(y_1|y_2)\in \mathcal{Y}$ such that $(y_1(u), y_2(v))=(q,r)$.
\item The $(\mathcal{P},\mathcal{F}_{\mathrm{weft}}, \mathcal{F}_{\mathrm{warp}} )$ 
          constructed as follows is an $(m,n,k,\ell)$-reticulation.
\begin{enumerate}
\item Let $\mathcal{P}=[m]\times [n]$.
\item For each $u\in[k]$, define a family $\mathcal{E}^u=\{E^u_q :  q\in [m]\}$ of weft lines, \\
         where $E^u_q = \{  (y_1(1),y_2(1))  :  (y_1|y_2)\in \mathcal{Y},\,  y_1(u) = q  \}$. 
         Let $\mathcal{F}_{\mathrm{weft}} = \{\mathcal{E}^u :  u\in [k] \}$.
\item For each $v\in [\ell]$, define a family $\mathcal{A}^v=\{A^v_r :  r\in[n]\}$ of warp lines,\\
         where $A^v_r  = \{ (y_1(1),y_2(1))  :  (y_1|y_2)\in \mathcal{Y},\, y_2(v) = r \}$. 
         Let $\mathcal{F}_{\mathrm{warp}} = \{\mathcal{A}^v :  v\in [\ell] \}$.     
\end{enumerate}
\item With reference to Definition \ref{def:constructarray}, if $\mathcal{Y}=\sigma(\mathcal{R})$, 
      then the reticulation constructed  from $\mathcal{Y}$ in {\textup{(ii)}} is the embedding $\theta(\mathcal{R})$ of $\mathcal{R}$ onto the grid $[m]\times [n]$ of Lemma \ref{lem:gridembed}.
\end{enumerate}
\end{theorem}

\begin{proof} 
Since $\mathcal{Y}$ is a svelte semi-orthogonal array, every element of $[m]\times[n]$, in particular $(q,r)$, occurs once among pairs  $(y_1(u), y_2(v))$ as $(y_1|y_2)$ runs over $\mathcal{Y}$. 
Thus, (i) holds.   
Now, (i) implies that there is a unique $(y_1|y_2)\in \mathcal{Y}$ in both $E^u_q$ and $A^v_r$. 

For the unique $(y_1|y_2)\in \mathcal{Y}$ with $y_1(u)=q$ and $y_2(v)=r$, we have $(y_1(1),y_2(1))\in E^u_q\cap A^v_r$.
If $(y'_1(1),y'_2(1))\in \mathcal{Y}$ is also in $E^u_q\cap A^v_r$, then $y'_1(u)=q$ and $y'_2(v)=r$, so $(y_1'|y_2')=(y_1|y_2)$ by (i).  
That is, (R-1) holds.
By (i), for each $(i,j)\in[m]\times [n]$, there is a unique $(y_1|y_2)\in \mathcal{Y}$ with $y_1(1)=i$, $y_2(1)=j$.  Say $y_1(u)=q$ and $y_2(v)=r$. 
Now $(i,j)\in E^u_q$ and $(i,j)\in A^v_r$, so all elements of $[m]\times [n]$ are in one of the lines of $\mathcal{E}^u$ and in one of the lines of $\mathcal{A}^v$.
Note that $(y_1|y_2)$ is the unique element of $Y$ with $y_1(u)=q$ and $y_2(1)=j$, so no other element $(y_1'|y_2')$ of $\mathcal{Y}$ with $(y_1'(1),y_2'(2))\in E^u_q$ has $y_2'(1)=j$.  
Thus, every element of $[m]\times [n]$ is in at most one line of $\mathcal{E}^u$.  Thus,  $\mathcal{E}^u$ is a partition of $[m]\times [n]$.  Similarly, 
$\mathcal{A}^v$ is a partition of $[m]\times [n]$.  Thus, (R-2) holds, giving (ii).
Part (iii) is straightforward.
\end{proof} 

From the preceding theorem, we get the following.

\begin{corollary}
The map $\sigma$ from the set of ordered $(m,n,k,\ell)$-reticulations of the grid to svelte semi-orthogonal arrays with parameters $(m,n,k,\ell)$ is invertible. 
\end{corollary}

We compare svelte semi-orthogonal arrays with orthogonal arrays and mixed orthogonal arrays.  

\begin{definition}
An \emph{orthogonal array $OA(N,k, m,t)$} with index $\lambda$ is an $N\times k$ array with entries from $[m]$ with the property that 
in any $N\times t$ subarray every element of $[m]^{t}$ occurs exactly $\lambda$ times as a row. Note $\lambda=N/m^t$. (See, for example, \cite[p.~2]{MR1693498}).
A \emph{mixed orthogonal array $OA(N, m^kn^\ell, t)$} is an $N\times (k+\ell)$ array in which the first $k$-many columns take values from $[m]$ and the last $\ell$-many columns take values from $[n]$ 
with the property that in any $N\times t$ subarray every possible pair occurs an equal number of times as a row.  
(This is a special case of a mixed orthogonal array; see, for example, \cite[p.~200]{MR1693498} for the general definition). 
\end{definition}

Unlike a mixed orthogonal array, a svelte semi-orthogonal array has no condition on columns in the same group.  This gives the following observation.

\begin{proposition}
A partition of the columns of an orthogonal array $OA(m^2, k+\ell, m, 2)$ into parts of size $k$ and $\ell$ 
is a svelte semi-orthogonal array with parameters $(m,m, k,\ell)$.
A mixed orthogonal array $OA(mn, m^kn^\ell, 2)$ is a svelte semi-orthogonal array with parameters $(m,n,k,\ell)$.
\end{proposition}

\section{Cooperative systems}
\label{sec:halflatinmatrices}

We  give a combinatorial characterization of the cooperative matrices $\mu(\mathcal{R})$ of an ordered reticulation $\mathcal{R}$.  We dub such objects cooperative systems.  
They resemble sets of mutually orthogonal Latin squares \cite[Chapter 5]{MR3495977}.  
We  examine connections to svelte semi-orthogonal arrays (and thus reticulations), showing a relationship with cooperative pairs analogous to the connections between certain  orthogonal arrays  and mutually orthogonal Latin squares \cite[Theorems 2.21]{MR1644242} and \cite[Section 11.1]{MR3495977}.   

\begin{definition}
A matrix $C\in \mathcal{M}_{m\times n}([m])$ is said to be \emph{column-Latin} whenever 
each element of $[m]$ appears exactly once in each column, that is, for each $j\in[n]$, $\{C(i,j) :  i\in[m]\}=[m]$.
A matrix is said to be \emph{row-Latin} whenever
each element of $[n]$ appears exactly once in each row, that is, for each $i\in[m]$, $\{R(i,j) :  j\in[n]\}=[n]$.
Matrices 
$C\in \mathcal{M}_{m\times n}([m])$ and $R\in \mathcal{M}_{m\times n}([n])$  are said to be \emph{orthogonal} whenever
every possible pairing of entries occurs exactly once in some position, that is, $\{(C(i,j), R(i,j)) :  (i,j)\in[m]\times [n]\}=[m]\times[n]$.
\end{definition}

These matrices are not to be confused with Latin rectangles, which are $m\times n$ rectangular arrays 
in which each of the numbers $1$, $2$, $\ldots$, $n$ occurs exactly once in each row and at most once in each column  (see \cite[Section 3.1]{MR3495977}).

\begin{definition}\label{def:cooppair}
A pair $(C,R)\in \mathcal{M}_{m\times n}([m])\times\mathcal{M}_{m\times n}([n])$
is said to be a \emph{cooperative pair} of order $m\times n$ whenever the following hold.
\begin{description}
\item[(CP-1)]  $C$ is column-Latin.
\item[(CP-2)]  $R$ is row-Latin.
\item[(CP-3)]  $C$ and $R$ are orthogonal.
\end{description}
By a \emph{cooperative system} with parameters $(m,n,k,\ell)$, we mean an array of  $m\times n$ matrices
$(C^1, \ldots, C^k; R^1, \ldots, R^\ell)$ such that $(C^u, R^v)$ is a cooperative pair for all $u\in[k]$ and $v\in [\ell]$.
\end{definition}

Cooperative pairs were introduced as an encoding of bireversible finite automata, and their relationship to $(m,n,2,2)$-reticulations  was first discussed in \cite{CurtinSavchuk:combBRA}.   
Let $\mathcal{A}=(Q, X, \pi,\lambda)$ be a finite Mealy automaton. 
Then  ${\mathcal{A}}$ is \emph{invertible} when for each fixed state $q\in Q$, the output function $\lambda(q,\cdot)$ induces a permutation of the letters $X$,
 in which case its inverse automaton is ${\mathfrak i}\mathcal{A}=(Q,X,\tilde{\pi},\tilde{\lambda})$, where
$\tilde{\pi}(q,x)=\pi(q,\lambda_q^{-1}(x))^{-1}$ and $\tilde{\lambda}(q,x)=\lambda_q^{-1}(x)$ for every $x\in X$ and $q\in Q$.
The \emph{dual} of $\mathcal{A}$ is a finite automaton ${\mathfrak d}\mathcal{A}=(X,Q,\hat{\lambda},\hat{\pi})$, where
$\hat\lambda(x,q)=\lambda(q,x)$ and $\hat\pi(x,q)=\pi(q,x)$ for every $x\in X$ and $q\in Q$.
The  automaton $\mathcal{A}$ is  \emph{reversible} if its dual is invertible. 
The automaton $\mathcal{A}$ is \emph{co-reversible} if for every $(r,y)\in Q\times X$ there is a  unique $(x,q)\in Q\times X$ such that 
$\pi(q,x)=r$ and $\lambda(q,x)=y$.
Finally, $\mathcal{A}$ is  \emph{bireversible} if it is invertible, and both ${\mathcal{A}}$ and ${\mathfrak i}\mathcal{A}$ are reversible.
See, for instance, \cite{MR1841755} for more on bireversible automata.

 \begin{theorem} \cite{CurtinSavchuk:combBRA}
 Let $\mathcal{A}=([m], [n], \pi,\lambda)$ be a finite Mealy automaton. 
 Define $C_{\mathcal{A}}\in \mathcal{M}_{m\times n}([m])$ and  $R_{\mathcal{A}}\in\mathcal{M}_{m\times n}([n])$ by 
 $C_{\mathcal{A}}(i,j)=\pi(i,j)$  and $R_{\mathcal{A}}(i,j)=\lambda(i,j)$.
Then $\mathcal{A}$\ is invertible if and only if $R_{\mathcal{A}}$ is row-Latin, 
         $\mathcal{A}$\ is reversible if and only if $C_{\mathcal{A}}$ is column-Latin, 
         $\mathcal{A}$\ is co-reversible if and only if $C_{\mathcal{A}}$ and $R_{\mathcal{A}}$ are orthogonal, and 
         $\mathcal{A}$\ is bireversible if and only if $(C_{\mathcal{A}}, R_{\mathcal{A}}) $ is a cooperative pair.
 \end{theorem}

\begin{example}
\label{ex:CVHR}
Let $H=H_{m\times n}\in  \mathcal{M}_{m\times n}([m])$ and $V=V_{m\times n}\in \mathcal{M}_{m\times n}([n])$ be the matrices with entries $H(i,j)=i$ and  $V(i,j)=j$.
Then for all column-Latin matrices $C\in  \mathcal{M}_{m\times n}([m])$ and all row-Latin matrices $R\in \mathcal{M}_{m\times n}([n])$,
 $(C, V)$ and $(H,R)$ are cooperative pairs. 
By Lemma \ref{lem:gridembed}, the matrices $H$ and $V$ appear in connection with reticulations via the construction of $C^1$ and $R^1$ in Definition \ref{def:constructarray}. 
\end{example}

\begin{lemma}
\label{lem:VHgivehalfLatin}
Let  $C\in \mathcal{M}_{m\times n}([m])$  and $R\in \mathcal{M}_{m\times n}([n])$, and let $H$, and $V$ be as in Example \ref{ex:CVHR}.
\begin{enumerate}
\item   $C$ is column-Latin if and only if $C$ and $V$ are orthogonal.
\item  $R$ is row-Latin  if and only if $R$ and $H$ are orthogonal.
\end{enumerate}
\end{lemma}

\begin{proof}
Column $j$ of $V$ is $[j\ j\ \dots\ j]^T$.  $C$ is column-Latin if and only if each value in $[m]$ appears exactly once in each column if and only if  every $(i,j)\in[m]\times[n]$ appears exactly once
among the pairs $(V(i,j), C(i,j))$.  A similar argument gives (ii).
\end{proof}

\begin{lemma}
Suppose $(C^1, \ldots, C^k; R^1, \ldots, R^\ell)$  is a cooperative system with parameters $(m,n,k,\ell)$. 
Then the array
$(H,C^1, \ldots, C^k; V, R^1, \ldots, R^\ell)$ is a cooperative system with parameters $(m,n,k+1, \ell+1)$. 
We say that a cooperative system is \emph{normalized} when $C^1=H$ and $R^1=V$.
\end{lemma} 

\begin{proof}
Clear from Example \ref{ex:CVHR} and Lemma \ref{lem:VHgivehalfLatin}.
\end{proof} 

A cooperative system can be normalized without adding repetition by prepending $H$  if it does not appear among the $C^u$ and prepending $V$ if it does not appear among the $R^v$.  
Otherwise, the lists can be reordered to bring $H$ and/or $V$ to the start of the list.    A cooperative system can be transformed into a  normalized cooperative system by applying an isotopism (Section \ref{sec:isotopy}), corresponding to choosing and ordering of the line families and lines of the related reticulation to define a grid of points (Lemma \ref{lem:gridembed}).

We re-use some notation in an obvious way.

\begin{theorem}
\label{thm:CS-TSOA}
Fix positive integers $m$, $n$, $k$, $\ell$.
\begin{enumerate}
\item Given a normalized cooperative system $\mathcal{C}=(C^1=H, \ldots, C^k; R^1=V, \ldots, R^\ell)$ with parameters $(m,n,k,\ell)$, let \\
$\sigma(\mathcal{C}) = \{ (C^1(i,j), \ldots, C^k(i,j); R^1(i,j), \ldots, R^\ell(i,j)) :  (i,j)\in[m]\times[n] \}$.
Then $\sigma(\mathcal{C})$ is a svelte semi-orthogonal array.
\item Given a svelte semi-orthogonal array $\mathcal{S}=\{(y_1|y_2)\}\subset[m]^k\times [n]^\ell$ with 
parameters $(m,n,k,\ell)$, let $\mu(\mathcal{S}) = (C^1, \ldots, C^k; R^1, \ldots, R^\ell)$, where
$C^u\in \mathcal{M}_{m\times n}([m])$ and $R^v\in \mathcal{M}_{m\times n}([n])$ satisfy
$C^u(i,j) =y_1(u)$ and $R^v(i,j)= y_2(v)$, where $(y_1|y_2)$ is the unique element of $\mathcal{S}$ with
$y_1(1)=i$ and $y_2(1)=j$.
Then $\mu(\mathcal{S})$ is a normalized cooperative system with parameters $(m,n,k,\ell)$.
\item $\sigma$ and $\mu$ are inverse bijections between the sets of svelte semi-orthogonal arrays and normalized cooperative systems with parameters $(m,n,k,\ell)$.
\end{enumerate}
\end{theorem}

\begin{proof} 
In (i), observe that for $(u,v)\in[k]\times[l]$, 
$\{ (C^u(i,j), R^v): (i,j)\in[m]\times [n]\}=[m]\times [n]$ since $C^u$ and $R^v$ are orthogonal.  Thus,  $\sigma(\mathcal{C})$ is a svelte semi-orthogonal array.
We now prove (ii).    Since $\mathcal{S}$ is a svelte semi-orthogonal array, $\{ (y_1(u), y_2(v)) : (y_1|y_2)\in \mathcal{S}\} =[m]\times[n] $ for all $(u,v)\in[k]\times[\ell]$.
In the case $(u,v)=(1,1)$,  the construction gives $C^1(i,j)=i$ and $R^1(i,j)=j$ for $(i,j)\in[m]\times[n]$, so $C^1=H$ and $R^1=V$.  More generally, it gives
that $\{ C^u(i,j), R^v(i,j) : y_1(1)=i, y_2(1)=j\}=[m]\times [n]$.   Thus, each pair $C^u$ and $R^v$ is orthogonal.  In light of Lemma \ref{lem:VHgivehalfLatin}, $C^u$ and $R^v$ form a cooperative pair.  
Hence,  $\mu(\mathcal{S})$ is a normalized cooperative system with parameters as stated in (ii).
The constructions of $\sigma$ and $\mu$ are clearly inverse operations, giving (iii). 
\end{proof} 

\begin{theorem}
Let  $C^u\in \mathcal{M}_{m\times n}([m])$ $(u\in[k])$ and $R^v\in \mathcal{M}_{m\times n}([n])$ $(v\in[\ell])$. Then 
the following are equivalent.
\begin{enumerate}
\item $(C^1, \ldots, C^k; R^1, \ldots, R^\ell)$ is the matrix array of an ordered $(m,n,k,\ell)$-reticulation.
\item $(C^1, \ldots, C^k; R^1, \ldots, R^\ell)$ is a normalized cooperative system.
\end{enumerate}
\end{theorem}

\begin{proof} 
Combine Theorems \ref{thm:TSOA-R} and \ref{thm:CS-TSOA}.
\end{proof}

\begin{example}
\label{ex:coopsys}
From Example \ref{ex:gridarraymatrix}, we have the cooperative system
\[\addtolength{\arraycolsep}{-0.4em}
(C^1=V, C^2; R^1=H, R^2) = 
 \left(
\left[\begin{array}{cccc}
1&1&1&1\\[-2pt]
2&2&2&2\\[-2pt]
3&3&3&3
\end{array}\right],\,  
\left[\begin{array}{cccc}
1&2&3&1\\[-2pt]
3&1&1&2\\[-2pt]
2&3&2&3
\end{array}\right]
;
\left[\begin{array}{cccc}
1&2&3&4\\[-2pt]
1&2&3&4\\[-2pt]
1&2&3&4
\end{array}\right],
\left[\begin{array}{cccc}
4&2&1&3\\[-2pt]
3&1&2&4\\[-2pt]
1&4&3&2
\end{array}\right]\right).
\]
\end{example}

We compare half-Latin matrix arrays with mutually orthogonal Latin squares.

\begin{definition}
A \emph{Latin square of order  $n$}  is an $n\times n$ array  such that every element of $[n]$  
appears exactly once in each row and each column.
Two Latin squares $L_1$ and $L_2$ of the same order $n$ are  \emph{orthogonal} whenever the pairs
$(L_1(i,j), L_2(i,j))$ are distinct for all  $(i,j)\in[n]\times[n]$. 
 A set of pairwise orthogonal Latin squares is called a set of \emph{mutually orthogonal Latin squares}.
 \end{definition}
 
References for Latin squares and sets of mutually orthogonal Latin squares include
\cite{MR1096296,MR3837138,MR3495977,MR1644242}.

\begin{theorem}\cite{MR3495977}
Let  $\mathcal{L}=\{L_1, \ldots, L_h\}$ be a set of mutually orthogonal Latin squares of order $n$.
Then $\mathcal{Y}=\{(i,j,L_1(i,j), L_2(i,j), \ldots, L_h (i,j))\}$ is an orthogonal array of strength 2 and index 1:
For any distinct $u$, $v\in[h+2]$, every element of $[n]\times [n]$ appears exactly once among 
the pairs $(y(u), y(v))$ as $y$ runs over $Y$.
\end{theorem}
 
As an immediate consequence, we get the following.

\begin{corollary}
Let $\mathcal{L}$ be a set of at least $(k+\ell)$-many mutually orthogonal Latin squares of order $n$.  
Let $\mathcal{L}_{\mathrm{weft}}=\{L^+_1, \ldots, L^+_k\}$ and 
      $\mathcal{L}_{\mathrm{warp}}=\{L^-_1, \ldots, L^-_\ell\}$ be disjoint subsets of $\mathcal{L}$.
Then 
\[  ( H, L^+_1, \ldots, L^+_k; V, L^-_1, \ldots, L^-_\ell) \]
is a cooperative system with parameters $(n,n, k+1, \ell+1)$.
\end{corollary}

We have an equivalence between ordered reticulations, svelte semi-orthogonal arrays, and cooperative systems which mirrors the equivalence between nets, orthogonal arrays, and sets of mutually orthogonal Latin squares. 
In fact, objects in the latter list are examples of objects in the former list.

\section{Semitransversal designs}
\label{sec:semitransversaldesign}

We describe a dual geometric structure of a reticulation, which provides another means of encoding the reticulation.  
Here, lines become varieties, taking the role of points, and points become blocks, taking the role of lines, while the families  of lines are collected into groups.
 We first recall  transversal designs, a special case of which are dual to   $(r,k)$-nets.   For more on transversal designs, see \cite{MR1742365, MR2246267}, for example.

\begin{definition}
A \emph{transversal design} with $r$ groups of size $k$ and index $\lambda$ is a triple $(\mathcal{V},\mathcal{G},\mathcal{B})$ satisfying
  $\mathcal{V}$ is a set of $rk$-many elements, called \emph{varieties};
  $\mathcal{G}$ is a partition of $\mathcal{V}$ into $r$-many cells of size $k$,  called \emph{groups};
  $\mathcal{B}$ is a collection of subsets of size $r$ of varieties, called \emph{blocks}; and
  each pair of varieties is either in exactly one group or exactly $\lambda$-many blocks.
Write $TD[r,\lambda;k]$ to denote such a transversal design.
\end{definition}

Given an $(r,k)$-net, a transversal design $TD[r;1;k]$ is defined as follows.
The varieties are the  $rk$-many lines of the net.  
The groups are the $r$-many parallel families  of lines.
There is a block for each of the  $k^2$-many points, consisting of all $r$ of the lines incident with the point.
Thus, by Proposition \ref{prop:net->retic}, any division of the groups of such a transversal design into two types encodes a reticulation.  
 
We formulate an analog of a transversal design $TD[r;1;k]$ associated with a reticulation.  
The groups come in two types, so we use the modifier ``semi''.  
Here we shall not attempt to formulate a general notation admitting an index/indices larger than one.
We use the adjective ``svelte'' to leave room for  a potential generalization. 

\begin{definition}
 By a \emph{svelte semitransversal design} with parameters $(m, n, k, \ell)$, we mean a tuple $(\mathcal{V}, \mathcal{G}_1, \mathcal{G}_2, \mathcal{B})$ 
  such that the following hold.
(i)    $\mathcal{V}$ is a set of $(km+\ell n)$-many \emph{varieties}.
(ii)   $\mathcal{G}_1$ is a set of $k$-many \emph{weft groups}, each consisting of $m$-many varieties, and 
       $\mathcal{G}_2$ is a set of $\ell$-many \emph{warp groups}, each consisting of $n$-many varieties such that $\mathcal{G}_1$ and $\mathcal{G}_2$ partition $\mathcal{V}$.
(iii)  $\mathcal{B}$ is a set of $mn$-many \emph{blocks}, each consisting of $(h+k)$-many varieties.  
(iv)  Every group meets every block in exactly one variety, and every pair of varieties from groups of different type appear together in exactly one block.
\end{definition}

We now show that svelte semitransversal designs are equivalent to reticulations.

\begin{lemma}
Given an $(m,n,k,\ell)$-reticulation $\mathcal{R}=(\mathcal{P}, \mathfrak{F}_\mathrm{weft}, \mathfrak{F}_\mathrm{warp})$, 
construct a svelte semitransversal design $(\mathcal{V}, \mathcal{G}_1, \mathcal{G}_2, \mathcal{B})$  with parameters $(m, n, k, \ell)$ as follows.
Let the varieties $\mathcal{V}$ be the set of all  $km$-many weft lines and all $\ell n$-many warp lines, 
      where  lines consisting of the same points are distinct varieties when they appear in different  families. 
Let the groups be $\mathcal{G}_1= \mathcal{F}_\mathrm{weft}$ and $\mathcal{G}_2=\mathcal{F}_\mathrm{warp}$.
For each $p\in \mathcal{P}$, define a block $B\in\mathcal{B}$  consisting of all lines incident with $p$.
\end{lemma}

\begin{proof}
The varieties, groups, and blocks are of the appropriate sizes by Theorem \ref{thm:regular}. 
By (R-1), every pair of varieties (lines) belonging to groups of different type are in exactly one block (meet in one point).
By (R-2), every group (family of lines) meets each block (point) in exactly one variety (line) (each family of lines partitions the points).
\end{proof}

\begin{lemma}
Given a svelte semitransversal design $(\mathcal{V}, \mathcal{G}_1, \mathcal{G}_2, \mathcal{B})$ with parameters $(m, n, k, \ell)$,
construct an $(m,n,k,\ell)$-reticulation  $\mathcal{R}=(\mathcal{P}, \mathcal{F}_\mathrm{weft}, \mathcal{F}_\mathrm{warp})$ as follows.
Let $\mathcal{P}=\mathcal{B}$.   
For each $G_u\in \mathcal{G}_1$ and each $V^u_j\in G_u$, let $E^u_i = \{B\in \mathcal{B} :  V^u_i\cap B\not=\emptyset \}$  $(u\in[k], i\in[m])$, and 
let $\mathcal{E}^u = \{ E^u_i \}_{i\in[m]}$ and $\mathcal{F}_\mathrm{weft} = \{\mathcal{E}^u\}_{u\in[k]}$.
For each $G_v\in \mathcal{G}_2$ and each $V^v_i\in G_u$, let $A^v_j= \{B\in \mathcal{B} :  V^v_j\cap B\not=\emptyset \}$   $(v\in[\ell], j\in[n])$, and 
let  $\mathcal{A}^v = \{ A^v_j \}_{j\in[n]}$ and $\mathcal{F}_\mathrm{warp} = \{\mathcal{A}^v\}_{v\in[\ell]}$.
\end{lemma}

\begin{proof}
Since every pair of varieties from different  sets appear together in exactly one group, any two lines of different type meet in a unique point (R-1).
Since every group meets every block in a unique point, every family of lines partitions $\mathcal{B}$ (R-2).  
Thus, $\mathcal{R}$ is a reticulation.  
By construction $|\mathcal{P}|=|\mathcal{B}|=mn$, 
                         $|\mathcal{F}_\mathrm{weft}|=k$, 
                         $|\mathcal{F}_\mathrm{warp}|=\ell$, 
                         each family of weft lines has $m$-many lines of size $n$, and 
                         each family of warp lines has $n$-many lines of size $m$.   
Thus, $\mathcal{R}$ has parameters $(m,n,k,\ell)$.
\end{proof}

\begin{example}
The svelte semitransversal design associated with the reticulation of Example  \ref{ex:gridarraymatrix} is as follows.
\[
\begin{array}{rcl}
\mathcal{V} \!\!&\!\!=\!\!&\!\! \{
 E^1_1, E^1_2, E^1_3, E^2_1, E^2_2, E^2_3,   A^1_1, A^1_2 , A^1_3, A^1_4,  A^2_1, A^2_2 , A^2_3, A^2_4
 \}, \\
 \mathcal{G}_1 \!\!&\!\!=\!\!&\!\!  \{ G^1_1 = \{   E^1_1, E^1_2, E^1_3\}, \, \,   G^1_2 = \{ E^2_1, E^2_2, E^2_3  \}  \},\\
  \mathcal{G}_2  \!\!&\!\!=\!\!&\!\!  \{ G^2_1 = \{  A^1_1, A^1_2 , A^1_3, A^1_4  \} \, \,   G^2_2  = \{  A^2_1, A^2_2 , A^2_3, A^2_4\} \},\\
\mathcal{B}  \!\!&\!\!=\!\!&\!\! 
\left\{\!\!\!
\begin{array}{c|c|c}
B_a=\{E^1_1, E^2_3, A^1_3, A^2_1\} & B_e=\{E^1_2, E^2_1, A^1_2, A^2_1\} & B_i=\{E^1_3, E^2_2, A^1_1, A^2_1\} \\
B_b=\{E^1_3, E^2_3, A^1_1, A^2_2\} & B_f=\{E^1_2, E^2_1, A^1_2, A^2_1\} & B_j=\{E^1_1, E^2_2, A^1_2, A^2_2\} \\
B_c=\{E^1_2, E^2_3, A^1_1, A^2_3\} & B_g=\{E^1_1, E^2_1, A^1_4, A^2_3\} & B_k=\{E^1_3, E^2_2, A^1_3, A^2_3\} \\
B_d=\{E^1_3, E^2_3, A^1_2, A^2_4\} & B_h=\{E^1_1, E^2_1, A^1_1, A^2_4\} & B_\ell=\{E^1_2, E^2_2, A^1_4, A^2_4\} \\
\end{array}
\!\!\!\right\}.
\end{array}
\]
\end{example}

\section{Parastrophy}
\label{sec:parastrophy}

Let $\mathcal{R}=(\mathcal{P}, \mathcal{F}_{\mathrm{weft}}, \mathcal{F}_{\mathrm{warp}})$ be a $(m,n,k,\ell)$-reticulation. The constructions in Definition \ref{def:constructarray} are relative to several orderings. 
In this section, we examine the effects of changing the ordering of the weft and warp families, the ordering of the weft families, and the ordering of the warp families.  
We adapt the notion of parastrophy of Latin squares (see \cite[Section 1.4]{MR3495977}) to cooperative pairs.  
The geometric interpretation of parastrophy in this section and isometry in the next section parallel those developed in  \cite[Chapter 2]{MR1125767} (see also \cite[Chapter 8]{MR3495977}, \cite[Section 1.4]{MR3837138}).

\begin{definition}
\label{def:parastrophy}
Let $\mathcal{Y}$ be a svelte semi-orthogonal array with parameters $(m,n,k,\ell)$. 
By a \emph{parastrophism}  of $\mathcal{Y}$ we mean an element of 
        $(\alpha, \beta, \gamma)\in \sym{2}\times \sym{k}\times \sym{\ell}$ 
acting on $[m]^k\times [n]^\ell$ as follows.
\[
((\alpha, \beta, \gamma)(y_1|y_2))(u,v)=
\begin{cases}
y_1(\beta^{-1}(u))y_2(\gamma^{-1}(v)) & \hbox{if $\alpha=(1)$}, \\
y_2(\gamma^{-1}(v))y_1(\beta^{-1}(u)) &  \hbox{if $\alpha=(12)$}.
\end{cases}
\]
Let $\mathcal{R}$ be a finite nondegenerate ordered $(m,n,k,\ell)$-reticulation.
By a \emph{parastrophism} of $\mathcal{R}$ we mean a parastrophism of $\sigma(\mathcal{R})$ acting as
$\sigma^{-1}((\alpha, \beta, \gamma)\sigma(\mathcal{R}))$ to produce another ordered reticulation.
We refer to the result of this action as a \emph{parastrophe} of $\mathcal{R}$.
\end{definition}

Observe that there are $(2\cdot k! \cdot \ell!)$-many parastrophisms, but not all of the resulting parastrophes need to be distinct.  
This is clearly the case for a reticulation with repetition, but even repetition-free reticulations may have distinct parastrophisms giving indistinct parastrophes.  
This is known to be the case for Latin squares, for instance.

Note that $\beta^{-1}$ and $\gamma^{-1}$ commute and are applied before $\alpha$. 
In particular,  $\beta^{-1}$ permutes the elements of $\mathcal{F}_{\mathrm{weft}}$,  $\gamma^{-1}$ permutes the elements of $\mathcal{F}_{\mathrm{warp}}$, and then $\alpha$ either fixes or swaps (transposes) weft and warp. 

\begin{lemma}
If $\mathcal{R}$ is an ordered $(m,n,k,\ell)$-reticulation  and $\pi$ is a parastrophism of $\mathcal{R}$, 
then $\pi(\mathcal{R})$ is a reticulation with parameter $(m,n,k,\ell)$ or $(n,m,\ell,k)$ depending upon whether the (multi)sets of warp and waft families are fixed or swapped.
\end{lemma}

\begin{proof} 
A parastrophism permutes the first $m$ components and the last $n$ components, possibly swapping the two groups.  
After doing so, it is still the case that for any choice of component from the first group and component from the second group all possible pairs appear. 
\end{proof} 

The inverses are introduced to give the following.

\begin{lemma}
Let $\mathcal{R}$ be an ordered $(m,n,k,\ell)$-reticulation.
Suppose $\pi=(\alpha, \beta, \gamma)$ is a parastrophism. 
\begin{enumerate}
\item If $\alpha$ fixes weft and warp, then
\begin{align}
\sigma(\pi(\mathcal{R})) &= \{ (C^{\beta(1)}(i,j)\ldots, C^{\beta(k)}(i,j);R^{\gamma(1)}(i,j), \ldots, R^{\gamma(\ell)}(i,j) ): i\in[m], j\in[n]\}, \nonumber\\
 \mu(\pi(\mathcal{R})) &= ( C^{\beta(1)}, \ldots, C^{\beta(k)};R^{\gamma(1)}, \ldots, R^{\gamma(\ell)}). \nonumber
\end{align}
\item If $\alpha$ swaps weft and warp, then 
\begin{align}
 \sigma(\pi(\mathcal{R})) &= \{ (R^{\gamma(1)}(j,i), \ldots, R^{\gamma(\ell)}(j,i); C^{\beta(1)}(j,i)\ldots, C^{\beta(k)}(j,i) ): i\in[m], j\in[n]\}, \nonumber\\
\mu(\pi(\mathcal{R})) &= ((R^{\gamma(1)})^\dag, \ldots, (R^{\gamma(\ell)})^\dag; (C^{\beta(1)})^\dag, \ldots, (C^{\beta(k)})^\dag). \nonumber
\end{align}
\end{enumerate}
\end{lemma}

\begin{proof} 
Straightforward from Definitions \ref{def:constructarray} and  \ref{def:parastrophy}.
\end{proof} 

\begin{example}
Apply the parastrophism $\pi=(\alpha=(12), \beta=(12), \gamma=(1))$ to the reticulation/svelte semi-orthogonal array  from Example \ref{ex:gridarraymatrix}:
 \[
\pi\!\left(\!\!\left\{\!\!\!\! \begin{array}{c}
 (1,1;1,4), (1,2;2,2), (1,3;3,1), (1,1;4,3)\\
 (2,3;1,3), (2,1;2,1), (2,1;3,2), (2,2;4,4) \\
 (3,2;1,1), (3,3;2,4), (3,2;3,3),  (3,3;4,2)
 \end{array}\!\!\!\!\right\}\!\!\right) 
\!\!\! = \!\!\!
 \left\{\!\!\!\! \begin{array}{cccc}
  (1,4;1,1),  (1,1;2,3), (1,3;3,2)\\
  (2,1;1,2),  (2,2; 2,1), (2,4;3,3)\\ 
  (3,2;1,2), (3,3;2,3), (3,1;3,1)\\
   (4,3;1,1), (4,4;2,2), (4,2;3,3)
 \end{array}\!\!\!\!\right\}\!.
 \]
The associated normalized cooperative system is $(C^1=H$, $C^2$; $R^1=V$, $R^2)$,  where  

$C^2=\left[\addtolength{\arraycolsep}{-0.4em}\begin{array}{cccc}
4&1&3\\[-2pt]
1&2&4\\[-2pt]
2&3&1\\[-2pt]
3&4&2\\
\end{array}\right]$, 
  and
$R^2=\left[\addtolength{\arraycolsep}{-0.4em}\begin{array}{cccc}
1&3&2\\[-2pt]
2&1&3\\[-2pt]
2&3&1\\[-2pt]
1&2&3\\
\end{array}\right]$. 
From the cooperative matrices one immediately recovers the reticulation.
The reticulation in Example \ref{ex:gridarraymatrix} has parameters $(3,4,2,2)$, so it has 8 parastrophes (including itself and the one described above). 
We omit the first weft and warp families which consist of horizontal and vertical lines, respectively.
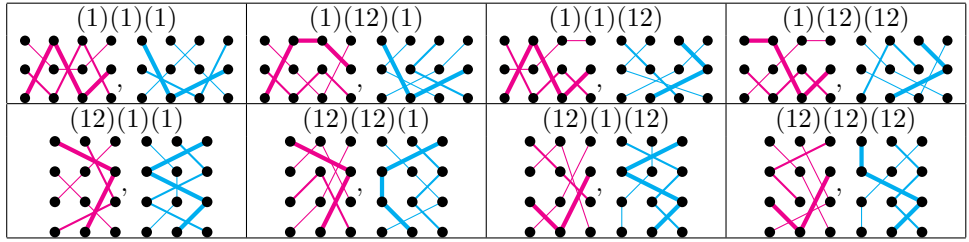
\[
\begin{array}{|c|c|c|c|}
\hline
(1)(1)(1) & (1)(12)(1) & (1)(1)(12) & (1)(12)(12)\\
\begin{tikzpicture}[scale=.38,baseline=3.5ex]
\draw[magenta] (1,3)--(2,2)--(3, 2 )--(4,3);
\draw[ultra thick, magenta] (1,1)--(2,3)--(3,1)--(4,2);
\draw[thick, magenta] (1,2)--(2,1)--(3,3)--(4,1);
\foreach \x in {1,2,3, 4} \foreach \y in {1,2, 3}
   \fill[black] (\x,\y) circle(5pt) ; 
\end{tikzpicture},
\ 
\begin{tikzpicture}[scale=.38,baseline=3.5ex]
\draw[ cyan] (1,1)--(2,2)--(3,3);
\draw[cyan] (2,3)--(3,2)--(4,1);
\draw[thick, cyan] (1,2)--(3,1)--(4,3);
\draw[ultra thick, cyan] (1,3)--(2,1)--(4,2);
\gridpoints{3}{4}
\end{tikzpicture}
&
\begin{tikzpicture}[scale=.38,baseline=3.5ex]
\draw[magenta] (1,3)--(2,2)--(3, 1 )--(4,3);
\draw[ultra thick, magenta] (1,1)--(2,3)--(3,3)--(4,2);
\draw[thick, magenta] (1,2)--(2,1)--(3,2)--(4,1);
\foreach \x in {1,2,3, 4} \foreach \y in {1,2, 3}
   \fill[black] (\x,\y) circle(5pt) ; 
\end{tikzpicture},
\ 
\begin{tikzpicture}[scale=.38,baseline=3.5ex]
\draw[ cyan] (1,1)--(3,2)--(4,3);
\draw[cyan] (2,3)--(1,2)--(3,1);
\draw[thick, cyan] (3,3)--(2,2)--(4,1);
\draw[ultra thick, cyan] (1,3)--(2,1)--(4,2);
\gridpoints{3}{4}
\end{tikzpicture}
&
\begin{tikzpicture}[scale=.38,baseline=3.5ex]
\draw[magenta] (1,2)--(2,2)--(3, 3 )--(4,3);
\draw[ultra thick, magenta] (1,1)--(2,3)--(3,1)--(4,2);
\draw[thick, magenta] (1,3)--(2,1)--(3,2)--(4,1);
\foreach \x in {1,2,3, 4} \foreach \y in {1,2, 3}
   \fill[black] (\x,\y) circle(5pt) ; 
\end{tikzpicture},
\ 
\begin{tikzpicture}[scale=.38,baseline=3.5ex]
\draw[ cyan] (1,1)--(3,2)--(4,3);
\draw[cyan] (2,3)--(1,2)--(4,1);
\draw[thick, cyan] (1,3)--(2,2)--(3,1);
\draw[ultra thick, cyan] (2,1)--(4,2)--(3,3);
\gridpoints{3}{4}
\end{tikzpicture}
&
\begin{tikzpicture}[scale=.38,baseline=3.5ex]
\draw[magenta] (1,1)--(2,2)--(3, 3 )--(4,3);
\draw[ultra thick, magenta] (1,3)--(2,3)--(3,1)--(4,2);
\draw[thick, magenta] (1,2)--(2,1)--(3,2)--(4,1);
\foreach \x in {1,2,3, 4} \foreach \y in {1,2, 3}
   \fill[black] (\x,\y) circle(5pt) ; 
\end{tikzpicture},
\ 
\begin{tikzpicture}[scale=.38,baseline=3.5ex]
\draw[ cyan] (1,2)--(3,1)--(4,3);
\draw[cyan] (1,3)--(2,2)--(4,1);
\draw[thick, cyan] (1,1)--(2,3)--(3,2);
\draw[ultra thick, cyan] (2,1)--(4,2)--(3,3);
\gridpoints{3}{4}
\end{tikzpicture}
\\

\hline
(12)(1)(1) & (12)(12)(1) & (12)(1)(12) & (12)(12)(12)\\
\ 
\begin{tikzpicture}[scale=.4,baseline=-7ex,x={(0,-1cm)}, y={(-1cm,0)}]
\draw[magenta] (1,1)--(2,2)--(3,3);
\draw[magenta] (2,3)--(3,2)--(4,1);
\draw[thick, magenta] (1,2)--(3,1)--(4,3);
\draw[ultra thick, magenta] (1,3)--(2,1)--(4,2);
\gridpoints{3}{4}
\end{tikzpicture},
\
\begin{tikzpicture}[scale=.4,baseline=-7ex, x={(0,-1cm)}, y={(-1cm,0)}]
\draw[cyan] (1,3)--(2,2)--(3, 2 )--(4,3);
\draw[ultra thick, cyan] (1,1)--(2,3)--(3,1)--(4,2);
\draw[thick, cyan] (1,2)--(2,1)--(3,3)--(4,1);
\foreach \x in {1,2,3, 4} \foreach \y in {1,2, 3}
   \fill[black] (\x,\y) circle(5pt) ; 
\end{tikzpicture}
&
\begin{tikzpicture}[scale=.4,baseline=-7ex, x={(0,-1cm)}, y={(-1cm,0)}]
\draw[magenta] (1,1)--(3,2)--(4,3);
\draw[magenta] (2,3)--(1,2)--(3,1);
\draw[thick, magenta] (3,3)--(2,2)--(4,1);
\draw[ultra thick, magenta] (1,3)--(2,1)--(4,2);
\gridpoints{3}{4}
\end{tikzpicture},
\
\begin{tikzpicture}[scale=.4,baseline=-7ex, x={(0,-1cm)}, y={(-1cm,0)}]
\draw[cyan] (1,3)--(2,2)--(3, 1 )--(4,3);
\draw[ultra thick, cyan] (1,1)--(2,3)--(3,3)--(4,2);
\draw[thick, cyan] (1,2)--(2,1)--(3,2)--(4,1);
\foreach \x in {1,2,3, 4} \foreach \y in {1,2, 3}
   \fill[black] (\x,\y) circle(5pt) ; 
\end{tikzpicture}
&
\begin{tikzpicture}[scale=.4,baseline=-7ex, x={(0,-1cm)}, y={(-1cm,0)}]
\draw[magenta] (1,1)--(3,2)--(4,3);
\draw[magenta] (2,3)--(1,2)--(4,1);
\draw[thick, magenta] (1,3)--(2,2)--(3,1);
\draw[ultra thick, magenta] (2,1)--(4,2)--(3,3);
\gridpoints{3}{4}
\end{tikzpicture},
\
\begin{tikzpicture}[scale=.4,baseline=-7ex, x={(0,-1cm)}, y={(-1cm,0)}]
\draw[cyan] (1,2)--(2,2)--(3, 3 )--(4,3);
\draw[ultra thick, cyan] (1,1)--(2,3)--(3,1)--(4,2);
\draw[thick, cyan] (1,3)--(2,1)--(3,2)--(4,1);
\foreach \x in {1,2,3, 4} \foreach \y in {1,2, 3}
   \fill[black] (\x,\y) circle(5pt) ; 
\end{tikzpicture}
&
\begin{tikzpicture}[scale=.4,baseline=-7ex, x={(0,-1cm)}, y={(-1cm,0)}]
\draw[magenta] (1,2)--(3,1)--(4,3);
\draw[magenta] (1,3)--(2,2)--(4,1);
\draw[thick, magenta] (1,1)--(2,3)--(3,2);
\draw[ultra thick, magenta] (2,1)--(4,2)--(3,3);
\gridpoints{3}{4}
\end{tikzpicture},
\
\begin{tikzpicture}[scale=.4,baseline=-7ex, x={(0,-1cm)}, y={(-1cm,0)}]
\draw[cyan] (1,1)--(2,2)--(3, 3 )--(4,3);
\draw[ultra thick, cyan] (1,3)--(2,3)--(3,1)--(4,2);
\draw[thick, cyan] (1,2)--(2,1)--(3,2)--(4,1);
\foreach \x in {1,2,3, 4} \foreach \y in {1,2, 3}
   \fill[black] (\x,\y) circle(5pt) ; 
\end{tikzpicture}
\\

\hline
\end{array}
\]
\end{example}

\section{Isotopy}
\label{sec:isotopy}

We now examine the impact of changing the orderings of lines within families. 
Although not changing the geometric structure of the reticulation, the representations as svelte semi-orthogonal arrays and cooperative systems are changed by easily understood permutations of rows, columns, and entries, similar to that for Latin squares/quasigroups \cite[Chapter 2]{MR1125767}.
We adapt the notion of isotopy for Latin squares (see \cite[Section 1.3]{MR3495977}) to cooperative pairs.

\begin{theorem}
   \label{thm:isotopy}
Let $\mathcal{R}=(\mathcal{P}, \mathcal{F}_{\mathrm{weft}}, \mathcal{F}_{\mathrm{warp}})$ be an 
ordered $(m,n,k,\ell)$-reticulation.  Change the orderings as described below to form a new ordered reticulation $\hat{\mathcal{R}}$. 
Then $\mu(\mathcal{R})$ and $\mu(\hat{\mathcal{R}})$ as constructed in Definition \ref{def:constructarray} are related as follows.

\begin{enumerate}
\item Let $\rho_1\in \sym{m}$.   Suppose $\hat{E}^1_i = E^1_{\rho_1(i)}$ $(i\in[m])$ and all other orderings are the same.
Then  $\hat{P}_{i,j} = P_{\rho_1(i),j}$,  $\hat{C}^1 = H$, $\hat{C}^u(i,j) = C^u(\rho_1(i),j)$ $(u\in[k]\setminus\{1\})$ and $\hat{R}^v(i,j) =R^v(\rho_1(i),j)$ $(v\in[\ell])$.
In other words, uniformly permute the rows of all $C^u$ and $R^v$ using $\rho_1$.
\item  Fix $u\in[k]\setminus\{1\}$.  Let $\rho_u \in  \sym{m}$.   
Suppose $\hat{E}^u_i = E^u_{\rho_u(i)}$ $(i\in[m])$ and all other orderings are the same.
Then   $\hat{C}^u(i,j) = \rho_u(C^u(i,j))$,  
           $\hat{C}^w=C^w$ $(w\in[k]\setminus\{u\})$, and 
           $\hat{R}^v=R^v$ $(v\in[\ell])$.
In other words, permute the entries of $C^u$ using $\rho_u$ and leave the other cooperative matrices unchanged.
\item Let $\kappa_1\in  \sym{n}$.   Suppose $\hat{A}^1_j = A^1_{\kappa_1(j)}$ $(j\in[n])$ and all other orderings are the same.
Then  $\hat{P}_{i,j} = P_{i,\kappa_1(j)}$, $\hat{C}^u(i,j) = C^u(i,\kappa_1(j))$ and $\hat{R}^v(i,j) =R^v(i,\kappa_1(j))$.
In other words, uniformly permute the columns of all $C^u$ and $R^v$ using $\kappa_1$.
\item  Fix $v\in[\ell]\setminus\{1\}$.  Let $\kappa_v\in \sym{n} $.     
Suppose $\hat{A}^1_j = A^1_{\kappa_v(j)}$ $(j\in[n])$ and all other orderings are the same. 
Then $\hat{R}^v(i,j) = \kappa_v(R^v(i,j))$,
         $\hat{R}^x =R^x$ $(x\in[\ell]\setminus\{v\})$, and 
         $\hat{C}^u = C^u$ $(u\in[k])$.
In other words, permute the entries of $R^v$ using $\kappa_v$  and leave the other cooperative matrices unchanged.
\end{enumerate}
\end{theorem}

\begin{proof}
In light of the transposition operation, it suffices to prove (i) and (ii). 

(i):  By construction, the point $P_{ij}$ is the intersection of $E^1_{i}$ and $A^1_{j}$. Thus,$P_{\rho_1(i),j}$ is the intersection of $\hat{E}^1_i=E^1_{\rho_1(i)}$ and $A^1_{j}$.  Hence, $\hat{P}_{i,j}= P_{\rho_1(i),j}$. 
By construction,  $C^u(i,j) = q$ when $P_{i,j}\in E^u_q$. When $u=1$, then $P_{i,j} \in E^1_i$, so $\hat{C}^1=H$.
When $u>1$, $\hat{C}^u(i,j) = C^u{(\rho_1(i),j)}$.
Similarly, $R^v(i,j) = r$ when $P_{i,j}\in A^v_r$, so $\hat{R}^v(i,j) = R^v{(\rho_1(i),j)}$.

(ii): Here, the indices of the points are unchanged, so $\hat{P}_{i,j}= P_{i,j}$.
Also,  $C^u(i,j) = q$ when $P_{i,j}\in E^u_q = E^u_{\rho_{u}(q)}$, so $\hat{C}^u(i,j) = \rho_{u}(q)$. All other cooperative matrices remain unchanged.
\end{proof}

\begin{definition}
Let  $\iota=(\rho_1, \ldots, \rho_k; \kappa_1, \ldots, \kappa_\ell)$ be a $(k+\ell)$-tuple with $\rho_u\in  \sym{m}$ and $\kappa_v\in  \sym{n}$.  
We call $\iota$ an \emph{isotopism} of ordered reticulations and cooperative pairs with parameters $(m,n,k,\ell)$ 
whenever $\iota$ acts on the families of lines $\mathcal{E}^u$ and $\mathcal{A}^v$ and matrices $C^u$ and $R^v$ as in Theorem \ref{thm:isotopy} for $u\in[k]$, $v\in[\ell]$.
\end{definition}

The constituent maps in an isotopism commute, so there is no ambiguity in applying components of an isotopism.

\begin{corollary}
Let $\mathcal{R}$ be an $(m,n,k,\ell)$-reticulation,  let $\iota=(\rho_1, \ldots, \rho_k; \kappa_1, \ldots, \kappa_\ell)$ be an isotopism of $\mathcal{R}$, and let $\hat{\mathcal{R}} = \iota(\mathcal{R})$.
Then the svelte semi-orthogonal arrays associated with  $\mathcal{R}$ and $\hat{\mathcal{R}}$  constructed in Definition \ref{def:constructarray} are related by 
\begin{align*}
\chi(\hat{\mathcal{R}})\ = \ 
\{ (\rho_1(x_1), \rho_2(x_2),\ldots, \rho_k(x_k); \kappa_1(y_1),  \kappa_2(y_2), \ldots, \kappa_\ell(y_\ell)) \phantom{XX} :& \\ 
  (x_1, \ldots, x_k; y_1, \ldots, y_\ell) \in  \chi(\mathcal{R})\}.&
\end{align*}
\end{corollary}

\begin{proof} 
Straightforward from Definition \ref{def:constructarray}  and Theorem \ref{thm:isotopy}.
\end{proof}

\begin{example}
With reference to Example \ref{ex:gridarraymatrix}, apply the permutation $\rho_1=(123)$ to $\mathcal{E}^1$:
$\mathcal{E}^1=\{ E_1^1=\{i,b,k,d\} , E_2^1=\{a,j,g,h\}, E_3^1=\{ e,f, c,\ell\}\}$ and leave all other orderings unchanged. 
Thus, 
$\hat{\mathcal{E}}^1=\{ \hat{E}_1^1=\{i,b,k,d\} , \hat{E}_2^1=\{a,j,g,h\}, \hat{E}_3^1=\{ e,f, c,\ell\}\}$
The rows of the grid are also permuted by $(123)$, so
$\hat{\mathcal{P}} =\left\{\addtolength{\arraycolsep}{-0.4em} \begin{array}{cccc}i&d&k&b\\[-2pt]h&j&a&g\\[-2pt] c&e&f&\ell\end{array}\right\}$.
Note that $\hat{C}^1=C^1$ and $\hat{R}^1=R^1$ and the rows of $C^2$ and $A^2$ are permuted by $(123)$:
\[ \hat{C}^1 =\left[\addtolength{\arraycolsep}{-0.4em}\begin{array}{cccc}
1&1&1&1\\[-2pt]
2&2&2&2\\[-2pt]
3&3&3&3
\end{array}\right], \ 
\hat{C}^2 =\left[\addtolength{\arraycolsep}{-0.4em}\begin{array}{cccc}
2&3&2&3\\[-2pt]
1&2&3&1\\[-2pt]
3&1&1&2
\end{array}\right],\ 
\hat{R}^1 =\left[\addtolength{\arraycolsep}{-0.4em}\begin{array}{cccc}
1&2&3&4\\[-2pt]
1&2&3&4\\[-2pt]
1&2&3&4
\end{array}\right],\ 
\hat{R}^2 =\left[\addtolength{\arraycolsep}{-0.4em}\begin{array}{cccc}
1&4&3&2\\[-2pt]
4&2&1&3\\[-2pt]
3&1&2&4
\end{array}\right].\] 
\end{example}

\begin{example}
With reference to Example \ref{ex:gridarraymatrix}, permute the lines of $\mathcal{E}^2$ with the permutation $\epsilon=(123)$:
$\hat{\mathcal{E}}^2= \epsilon(\mathcal{E}^2) =\{ \hat{E}_1^2 = E_{\epsilon^{-1}(1)}^2= \{i,j,k,\ell\},
                             \hat{E}_2^2 = E_{\epsilon^{-1}(2)}^2=\{a,b,c,d\}, 
                             \hat{E}_3^2 = E_{\epsilon^{-1}(3)}^2=\{e,f,g,h\} \}$ and leave all other orderings unchanged. 
Note that $\hat{\mathcal{P}}=\mathcal{P}$, so 
$\hat{C}^2 = \left[\addtolength{\arraycolsep}{-0.4em}\begin{array}{cccc}
3&1&2&3\\[-2pt]
2&3&3&1\\[-2pt]
1&2&1&2
\end{array}\right]$, 
 formed by permuting the entries of $C^2$ by $(132)$.
\end{example}

\begin{corollary}
Let $\mathcal{C}=(C^1, \ldots, C^k; R^1, \ldots, R^\ell)$ be a normalized cooperative system with parameters $(m,n,k,\ell)$, and 
    let $\iota=(\rho_1, \ldots, \rho_k; \kappa_1, \ldots, \kappa_\ell)$ be an isotopism of  $\mathcal{C}$.
Then $\iota(\mathcal{C})$ is a normalized cooperative system. 
\end{corollary}

\begin{proof}
Note that all isotopism operations leave $\hat{C}^1=H$ and $\hat{R}^1=V$ invariant.
\end{proof}

\section{Reticulations with a small parameter}
\label{sec:smallparam}

We describe finite nondegenerate repetition-free reticulations with just one or two weft lines.  
We also describe the finite nondegenerate repetition-free reticulations with just one family of weft lines.
Using the transpose operation to swap the roles of weft and warp gives the corresponding results for warp lines.

\begin{lemma}
Let $\mathcal{R}$ be a  repetition-free nondegenerate finite reticulation.
\begin{enumerate}
\item If every weft line  of $\mathcal{R}$ has size 1, then $\mathcal{R}$ has parameters $(1, n, 1,1)$.
\item If every point lies in the same weft line, then $\mathcal{R}$ has parameters $(m,1,1,1)$.
\end{enumerate}
\end{lemma}

\begin{proof} 
In (i), every warp line must meet each of the weft lines of size 1 by (R-1). 
In particular, there is a unique warp line, so it contains every point. 
Since $\mathcal{R}$ is repetition-free, the parameters are $(1, n, 1,1)$.
In (ii), each warp line contains just one point from the unique weft line by (R-1).
Since $\mathcal{R}$ is repetition-free, the parameters are $(m,1,1,1)$.
\end{proof} 

\begin{example}
We draw reticulations with parameters $(1, 3, 1,1)$ and $(3,1,1,1)$:

\centering

$\left(\left\{\begin{tikzpicture}[scale=.5,baseline=5.5ex]
\foreach \y in {1,2, 3}
{ \draw[ red] (.6,\y)--(1.4,\y); };
\gridpoints{3}{1}
\end{tikzpicture}\right\},
\ 
\left\{ \begin{tikzpicture}[scale=.5,baseline=5.5ex]
\draw[ blue] (1,1)--(1, 3); 
\gridpoints{3}{1}
\end{tikzpicture}\right\}\right)$,
\qquad
$\left(\{\begin{tikzpicture}[scale=.5,baseline=2.5ex]
\draw[ red] (1,1)--(3, 1); 
\gridpoints{1}{3}
\end{tikzpicture}\},
\ 
\{ \begin{tikzpicture}[scale=.5,baseline=2.5ex]
\foreach \y in {1,2, 3}
{ \draw[ blue] (\y,.6)--(\y,1.4); };
\gridpoints{1}{3}
\end{tikzpicture}\}\right)$.
\end{example}

For cooperative systems with parameters $(2,n,k,\ell)$, we adapt a counting argument show to us by our coauthor of \cite{CurtinSavchuk:combBRA}.

\begin{lemma}
Let $n$ be a positive integer.
\begin{enumerate}
\item For each choice of nonnegative integers $q$ and $r$ with $q+r\leq n$ and
disjoint subsets $N_1$ and $N_2$ of $[n]$ with respective sizes $q$ and $r$ with $1\in N_1$ and $N_2\not=\emptyset$, the following is 
a repetition-free $(2,n,2^{n-q-r}+1,q!r!)$-reticulation on $[2]\times[n]$.
   \begin{enumerate}
      \item Let $\mathcal{N}$ be the set of all $n$-tuples which are $1$ in all components in $N_1$ and which are $2$ in components in $N_2$ together with the all-ones tuple. 
      \item For $T\in \mathcal{N}$, define a  family of weft lines by
     $\mathcal{E}^T:=\{(T(j), j) :  j\in[n]\}$ and $\{(3-T(j)), j) :  j\in[n]\}$.
     (The all-ones tuple gives the family of horizontal lines.)
     \item For each permutation $\sigma_1$ of $N_1$ and permutation $\sigma_2$ of $N_2$, set $\sigma=\sigma_1\sigma_2\in  \sym{n}$, and
define a family of warp lines  by \\
$\mathcal{A}^\sigma:=\{\{(1,j), (2, \sigma(j))\} :  j \in[n]\}$. 
(The identity permutation gives the family of vertical lines.)
\end{enumerate}
\item Let $\mathcal{R}$ be a repetition-free  $(2,n,k,\ell)$-reticulation.
         Order $\mathcal{R}$ and embed it into $[m]\times [n]$ 
         via the map $\theta$ of Lemma \ref{lem:gridembed}.
         Then $\theta(\mathcal{R})$ is contained in one of the reticulations described in  {\textup{(i)}}.
\end{enumerate}
\end{lemma}

\begin{proof} 
(i): By construction, each family of lines partitions $[m]\times[n]$ and each weft line meets each weft line in a unique point.

(ii): Order $\mathcal{R}$ and embed it in $[m]\times[n]$ via the map $\theta$ of Lemma \ref{lem:gridembed}.  To each weft family, associate an $n$-tuple whose $j$-th component 
is $1$ if $(1,1)$ and $(1,j)$ are on the same line and $2$ if $(2,1)$ and $(1,j)$ are on the same line.  
The weft family consisting of horizontal lines corresponds to the all-ones tuple.   Take $N_1$ to be the set of
components where all the  constructed tuples have value 1 and take $N_2$ to be the set of components where all constructed tuples have value 2. 
The construction of families of weft lines in (i) returns the family used to define the tuple. 
Since the warp lines consist of two points on different lines, they consist of points $(i,1)$ and $(i',2)$ where components $i$ and $i'$ of $T$ take the same value for all $T$. 
Now the map $\sigma$ taking $i$ to $i'$  in this case and fixing $i''$ when there are two $T$'s which do not agree in component $i''$.  
The construction of warp lines from this $\sigma$ in (i) returns the initial weft line.
\end{proof} 

\begin{example}
For $n=6$, $q=2$, $r=2$, $J_1=\{1,2\}$, $J_2=\{3,4\}$, we have a reticulation with parameters
$(2,6,2^{n-q-r}+1=5, q!r!=4)$.

\centering
\begin{minipage}{.4\textwidth}
$\begin{array}{|c|c|}
\hline
T &            \hbox{weft lines} \\ \hline  & \\[-1.6ex]
(1,1,1,1,1,1)& 
\begin{tikzpicture}[scale=.4,baseline=2ex]
\draw[ red] (1,1)--(2,1)--(3, 1)--(4,1)--(5,1)--(6,1); 
\draw[ red] (1,2)--(2,2)--(3, 2)--(4,2)--(5,2)--(6,2);  
\gridpoints{2}{6}
\end{tikzpicture}\\ \hline & \\[-1.6ex]
(1,1,2,2,1,1)&
\begin{tikzpicture}[scale=.4,baseline=2ex]
\draw[ red] (1,1)--(2,1)--(3, 2)--(4,2)--(5,1)--(6,1); 
\draw[ red] (1,2)--(2,2)--(3, 1)--(4,1)--(5,2)--(6,2);  
\gridpoints{2}{6}
\end{tikzpicture}\\ \hline & \\[-1.6ex]
(1,1,2,2,1,2)&
\begin{tikzpicture}[scale=.4,baseline=2ex]
\draw[ red] (1,1)--(2,1)--(3, 2)--(4,2)--(5,1)--(6,2); 
\draw[ red] (1,2)--(2,2)--(3, 1)--(4,1)--(5,2)--(6,1);  
\gridpoints{2}{6}
\end{tikzpicture}\\ \hline & \\[-1.6ex]
(1,1,2,2,2,1)&
\begin{tikzpicture}[scale=.4,baseline=2ex]
\draw[ red] (1,1)--(2,1)--(3, 2)--(4,2)--(5,2)--(6,1); 
\draw[ red] (1,2)--(2,2)--(3, 1)--(4,1)--(5,1)--(6,2);  
\gridpoints{2}{6}
\end{tikzpicture}\\ \hline & \\[-1.6ex]
(1,1,2,2,2,2)&
\begin{tikzpicture}[scale=.4,baseline=2ex]
\draw[ red] (1,1)--(2,1)--(3, 2)--(4,2)--(5,2)--(6,2); 
\draw[ red] (1,2)--(2,2)--(3, 1)--(4,1)--(5,1)--(6,1);  
\gridpoints{2}{6}
\end{tikzpicture}\\ \hline 
\end{array}$
\end{minipage}
\begin{minipage}{.4\textwidth}
$\begin{array}{|c|c|}
\hline
\sigma_1\sigma_2 & \hbox{warp lines}\\ \hline & \\[-1.6ex]
(1)(2)(3)(4) &
\begin{tikzpicture}[scale=.4,baseline=2ex]
\draw[ blue] (1,1)--(1,2);  \draw[ blue] (2,1)--(2,2);
\draw[ blue] (3,1)--(3,2);  \draw[ blue] (4,1)--(4,2);
\draw[ blue] (5,1)--(5,2);  \draw[ blue] (6,1)--(6,2);
\gridpoints{2}{6}
\end{tikzpicture}\\ \hline & \\[-1.6ex]
(1)(2)(34) & 
 \begin{tikzpicture}[scale=.4,baseline=2ex]
\draw[ blue] (1,1)--(1,2);  \draw[ blue] (2,1)--(2,2);
\draw[ blue] (3,1)--(4,2);  \draw[ blue] (4,1)--(3,2);
\draw[ blue] (5,1)--(5,2);  \draw[ blue] (6,1)--(6,2);
\gridpoints{2}{6}
\end{tikzpicture}\\ \hline & \\[-1.6ex]
(12)(3)(4) &
\begin{tikzpicture}[scale=.4,baseline=2ex]
\draw[ blue] (1,1)--(2,2);  \draw[ blue] (2,1)--(1,2);
\draw[ blue] (3,1)--(3,2);  \draw[ blue] (4,1)--(4,2);
\draw[ blue] (5,1)--(5,2);  \draw[ blue] (6,1)--(6,2);
\gridpoints{2}{6}
\end{tikzpicture}\\ \hline & \\[-1.6ex]
(12)(34) & 
\begin{tikzpicture}[scale=.4,baseline=2ex]
\draw[ blue] (1,1)--(2,2);  \draw[ blue] (2,1)--(1,2);
\draw[ blue] (3,1)--(4,2);  \draw[ blue] (4,1)--(3,2);
\draw[ blue] (5,1)--(5,2);  \draw[ blue] (6,1)--(6,2);
\gridpoints{2}{6}
\end{tikzpicture}\\ \hline 
\end{array}$
\end{minipage}
\end{example}

The last easy case is where there is just one family of weft (or warp) lines.

\begin{lemma}
\label{lem:mn1*}
Let $m$, $n$ be positive integers.
\begin{enumerate}
\item \label{lem:mn1const}
   The following is a repetition-free $(m,n,1, (n!)^{m-1})$-reticulation on $[m]\times[n]$.
   \begin{enumerate}
   \item The unique family of weft lines consists of horizontal lines $\{ \{(i,j)  :  j\in[n] \}  :  i\in[m]\}$.
   \item  For each choice of  $(m-1)$-many permutations $\sigma_1$, $\ldots$, $\sigma_{m-1}\in  \sym{n}$,
            form a family $\mathcal{F}_{\sigma_1, \ldots, \sigma_{m-1}}$ of warp lines on $[m]\times [n]$ as the set of lines\\
             $\{ \{(1,i),(2,\sigma_1(i)), \ldots, (m,\sigma_{m-1}(i))\}  :  i\in [n]\}$.     
   \end{enumerate}
\item Let $\mathcal{R}$ be a repetition-free $(m,n,1,\ell)$-reticulation.
         Order $\mathcal{R}$ and embed it into $[m]\times [n]$ via the map $\theta$ of Lemma \ref{lem:gridembed}.
         Then $\theta(\mathcal{R})$ is contained in the reticulation of {\textup{(i)}}.
\end{enumerate}
\end{lemma}

\begin{proof} 
(i): The single warp family partitions the grid.  To meet each of the weft lines in exactly one point, 
the warp lines need only pick exactly one point from each row of the grid.  Each of the points in the first row is in a distinct line.  
The construction uses the permutation $\sigma_i$ to specify which point of row $i+1$ is in the same warp line as each point of row $i$. 
The fact that the $\sigma_i$ are permutations ensures that  these lines partition the grid.   
There are $(n!)^{m-1}$ ways to construct such a family, giving a repetition-free  $(m,n,1, (n!)^{m-1})$-reticulation.

(ii): The construction of Definition \ref{def:constructarray} gives each point indices 
$(i,j)\in[m]\times [n]$ based on a choice of one weft and warp families and the choice of ordering of these two families by Lemma \ref{lem:gridembed}. 
Given the embedding of a warp family onto the grid, define $\sigma_i(j)=j'$ whenever the point $(i,j')$ is in the same line as the point $(1, j)$. 
The construction of (i) returns the embedding. Thus, (ii) holds. 
\end{proof} 

 We have been drawing the lines in our figures by using permutations differently.
 For each choice of permutations $\sigma_1$, $\ldots$, $\sigma_{m-1}\in  \sym{n}$,
            form a family of warp lines on $[m]\times [n]$ by 
             $\{ \{(1,j),  (2,\sigma_1(j)), (m,\sigma_2\circ\sigma_1(j), \ldots, 
                                           \sigma_{m-1}\circ\cdots\circ \sigma_1(j))\}  : j\in[n]\}$.
Given the embedding of a warp family onto the grid, define $\sigma_i(j)=j'$ whenever $(i,j)$ and $(i+1, j')$ are in the same line.

\begin{example}
Lemma \ref{lem:mn1*} gives a reticulation with parameters
$(3,3,1,(n!)^{m-1}=36)$: 

\noindent
$\left\{\!
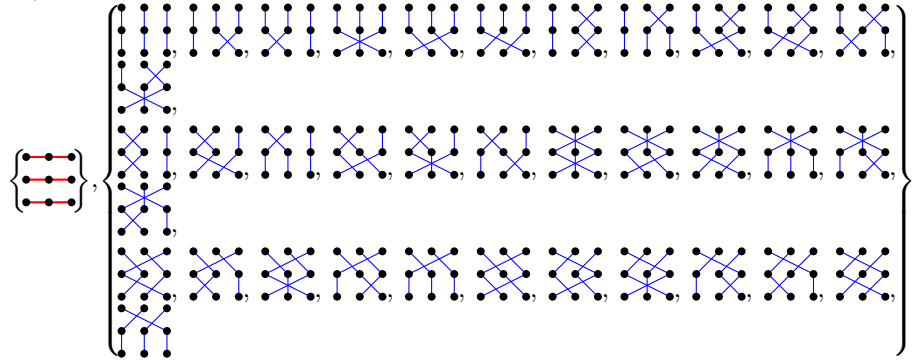
 \begin{tikzpicture}[scale=.30,baseline=3.3ex]
\foreach \y in {1,2, 3}
{ \draw[thick, red] (1,\y)--(3,\y); };
\gridpoints{3}{3} 
\end{tikzpicture}\!\right\}$, 
$\!\!\left\{\!
\begin{minipage}{.85\textwidth}
\begin{tikzpicture}[scale=.30,baseline=2ex]
\draw[ blue] (1,1)--(1,2)--(1, 3); 
\draw[ blue] (2,1)--(2,2)--(2, 3); 
\draw[ blue] (3,1)--(3,2)--(3, 3); 
\gridpoints{3}{3}
\end{tikzpicture},
\begin{tikzpicture}[scale=.30,baseline=2ex]
\draw[ blue] (1,1)--(1,2)--(1, 3); 
\draw[ blue] (3,1)--(2,2)--(2, 3); 
\draw[ blue] (2,1)--(3,2)--(3, 3); 
\gridpoints{3}{3}
\end{tikzpicture},
\begin{tikzpicture}[scale=.30,baseline=2ex]
\draw[ blue] (2,1)--(1,2)--(1, 3); 
\draw[ blue] (1,1)--(2,2)--(2, 3); 
\draw[ blue] (3,1)--(3,2)--(3, 3); 
\gridpoints{3}{3}
\end{tikzpicture},
\begin{tikzpicture}[scale=.30,baseline=2ex]
\draw[ blue] (3,1)--(1,2)--(1, 3); 
\draw[ blue] (2,1)--(2,2)--(2, 3); 
\draw[ blue] (1,1)--(3,2)--(3, 3); 
\gridpoints{3}{3}
\end{tikzpicture},
\begin{tikzpicture}[scale=.30,baseline=2ex]
\draw[ blue] (2,1)--(1,2)--(1, 3); 
\draw[ blue] (3,1)--(2,2)--(2, 3); 
\draw[ blue] (1,1)--(3,2)--(3, 3); 
\gridpoints{3}{3}
\end{tikzpicture},
\begin{tikzpicture}[scale=.30,baseline=2ex]
\draw[ blue] (3,1)--(1,2)--(1, 3); 
\draw[ blue] (1,1)--(2,2)--(2, 3); 
\draw[ blue] (2,1)--(3,2)--(3, 3); 
\gridpoints{3}{3}
\end{tikzpicture},
\begin{tikzpicture}[scale=.30,baseline=2ex]
\draw[ blue] (1,1)--(1,2)--(1, 3); 
\draw[ blue] (2,1)--(3,2)--(2, 3); 
\draw[ blue] (3,1)--(2,2)--(3, 3); 
\gridpoints{3}{3}
\end{tikzpicture},
\begin{tikzpicture}[scale=.30,baseline=2ex]
\draw[ blue] (1,1)--(1,2)--(1, 3); 
\draw[ blue] (3,1)--(3,2)--(2, 3); 
\draw[ blue] (2,1)--(2,2)--(3, 3); 
\gridpoints{3}{3}
\end{tikzpicture},
\begin{tikzpicture}[scale=.30,baseline=2ex]
\draw[ blue] (2,1)--(1,2)--(1, 3); 
\draw[ blue] (1,1)--(3,2)--(2, 3); 
\draw[ blue] (3,1)--(2,2)--(3, 3); 
\gridpoints{3}{3}
\end{tikzpicture},
\begin{tikzpicture}[scale=.30,baseline=2ex]
\draw[ blue] (3,1)--(1,2)--(1, 3); 
\draw[ blue] (2,1)--(3,2)--(2, 3); 
\draw[ blue] (1,1)--(2,2)--(3, 3); 
\gridpoints{3}{3}
\end{tikzpicture},
\begin{tikzpicture}[scale=.30,baseline=2ex]
\draw[ blue] (2,1)--(1,2)--(1, 3); 
\draw[ blue] (3,1)--(3,2)--(2, 3); 
\draw[ blue] (1,1)--(2,2)--(3, 3); 
\gridpoints{3}{3}
\end{tikzpicture},
\begin{tikzpicture}[scale=.30,baseline=2ex]
\draw[ blue] (3,1)--(1,2)--(1, 3); 
\draw[ blue] (1,1)--(3,2)--(2, 3); 
\draw[ blue] (2,1)--(2,2)--(3, 3); 
\gridpoints{3}{3}
\end{tikzpicture},

\smallskip
\begin{tikzpicture}[scale=.30,baseline=2ex]
\draw[ blue] (1,1)--(2,2)--(1, 3); 
\draw[ blue] (2,1)--(1,2)--(2, 3); 
\draw[ blue] (3,1)--(3,2)--(3, 3); 
\gridpoints{3}{3}
\end{tikzpicture},
\begin{tikzpicture}[scale=.30,baseline=2ex]
\draw[ blue] (1,1)--(2,2)--(1, 3); 
\draw[ blue] (3,1)--(1,2)--(2, 3); 
\draw[ blue] (2,1)--(3,2)--(3, 3); 
\gridpoints{3}{3}
\end{tikzpicture},
\begin{tikzpicture}[scale=.30,baseline=2ex]
\draw[ blue] (2,1)--(2,2)--(1, 3); 
\draw[ blue] (1,1)--(1,2)--(2, 3); 
\draw[ blue] (3,1)--(3,2)--(3, 3); 
\gridpoints{3}{3}
\end{tikzpicture},
\begin{tikzpicture}[scale=.30,baseline=2ex]
\draw[ blue] (3,1)--(2,2)--(1, 3); 
\draw[ blue] (2,1)--(1,2)--(2, 3); 
\draw[ blue] (1,1)--(3,2)--(3, 3); 
\gridpoints{3}{3}
\end{tikzpicture},
\begin{tikzpicture}[scale=.30,baseline=2ex]
\draw[ blue] (2,1)--(2,2)--(1, 3); 
\draw[ blue] (3,1)--(1,2)--(2, 3); 
\draw[ blue] (1,1)--(3,2)--(3, 3); 
\gridpoints{3}{3}
\end{tikzpicture},
\begin{tikzpicture}[scale=.30,baseline=2ex]
\draw[ blue] (3,1)--(2,2)--(1, 3); 
\draw[ blue] (1,1)--(1,2)--(2, 3); 
\draw[ blue] (2,1)--(3,2)--(3, 3); 
\gridpoints{3}{3} 
\end{tikzpicture},
\begin{tikzpicture}[scale=.30,baseline=2ex]
\draw[ blue] (1,1)--(3,2)--(1, 3); 
\draw[ blue] (2,1)--(2,2)--(2, 3); 
\draw[ blue] (3,1)--(1,2)--(3, 3); 
\gridpoints{3}{3}
\end{tikzpicture},
\begin{tikzpicture}[scale=.30,baseline=2ex]
\draw[ blue] (1,1)--(3,2)--(1, 3); 
\draw[ blue] (3,1)--(2,2)--(2, 3); 
\draw[ blue] (2,1)--(1,2)--(3, 3); 
\gridpoints{3}{3}
\end{tikzpicture},
\begin{tikzpicture}[scale=.30,baseline=2ex]
\draw[ blue] (2,1)--(3,2)--(1, 3); 
\draw[ blue] (1,1)--(2,2)--(2, 3); 
\draw[ blue] (3,1)--(1,2)--(3, 3); 
\gridpoints{3}{3}
\end{tikzpicture},
\begin{tikzpicture}[scale=.30,baseline=2ex]
\draw[ blue] (3,1)--(3,2)--(1, 3); 
\draw[ blue] (2,1)--(2,2)--(2, 3); 
\draw[ blue] (1,1)--(1,2)--(3, 3); 
\gridpoints{3}{3}
\end{tikzpicture},
\begin{tikzpicture}[scale=.30,baseline=2ex]
\draw[ blue] (2,1)--(3,2)--(1, 3); 
\draw[ blue] (3,1)--(2,2)--(2, 3); 
\draw[ blue] (1,1)--(1,2)--(3, 3); 
\gridpoints{3}{3}
\end{tikzpicture},
\begin{tikzpicture}[scale=.30,baseline=2ex]
\draw[ blue] (3,1)--(3,2)--(1, 3); 
\draw[ blue] (1,1)--(2,2)--(2, 3); 
\draw[ blue] (2,1)--(1,2)--(3, 3); 
\gridpoints{3}{3}
\end{tikzpicture},

\smallskip
\begin{tikzpicture}[scale=.30,baseline=2ex]
\draw[ blue] (1,1)--(2,2)--(1, 3); 
\draw[ blue] (2,1)--(3,2)--(2, 3); 
\draw[ blue] (3,1)--(1,2)--(3, 3); 
\gridpoints{3}{3}
\end{tikzpicture},
\begin{tikzpicture}[scale=.30,baseline=2ex]
\draw[ blue] (1,1)--(2,2)--(1, 3); 
\draw[ blue] (3,1)--(3,2)--(2, 3); 
\draw[ blue] (2,1)--(1,2)--(3, 3); 
\gridpoints{3}{3}
\end{tikzpicture},
\begin{tikzpicture}[scale=.30,baseline=2ex]
\draw[ blue] (2,1)--(2,2)--(1, 3); 
\draw[ blue] (1,1)--(3,2)--(2, 3); 
\draw[ blue] (3,1)--(1,2)--(3, 3); 
\gridpoints{3}{3}
\end{tikzpicture},
\begin{tikzpicture}[scale=.30,baseline=2ex]
\draw[ blue] (3,1)--(2,2)--(1, 3); 
\draw[ blue] (2,1)--(3,2)--(2, 3); 
\draw[ blue] (1,1)--(1,2)--(3, 3); 
\gridpoints{3}{3}
\end{tikzpicture},
\begin{tikzpicture}[scale=.30,baseline=2ex]
\draw[ blue] (2,1)--(2,2)--(1, 3); 
\draw[ blue] (3,1)--(3,2)--(2, 3); 
\draw[ blue] (1,1)--(1,2)--(3, 3); 
\gridpoints{3}{3}
\end{tikzpicture},
\begin{tikzpicture}[scale=.30,baseline=2ex]
\draw[ blue] (3,1)--(2,2)--(1, 3); 
\draw[ blue] (1,1)--(3,2)--(2, 3); 
\draw[ blue] (2,1)--(1,2)--(3, 3); 
\gridpoints{3}{3}
\end{tikzpicture},
\begin{tikzpicture}[scale=.30,baseline=2ex]
\draw[ blue] (1,1)--(3,2)--(1, 3); 
\draw[ blue] (2,1)--(1,2)--(2, 3); 
\draw[ blue] (3,1)--(2,2)--(3, 3); 
\gridpoints{3}{3}
\end{tikzpicture},
\begin{tikzpicture}[scale=.30,baseline=2ex]
\draw[ blue] (1,1)--(3,2)--(1, 3); 
\draw[ blue] (3,1)--(1,2)--(2, 3); 
\draw[ blue] (2,1)--(2,2)--(3, 3); 
\gridpoints{3}{3}
\end{tikzpicture},
\begin{tikzpicture}[scale=.30,baseline=2ex]
\draw[ blue] (2,1)--(3,2)--(1, 3); 
\draw[ blue] (1,1)--(1,2)--(2, 3); 
\draw[ blue] (3,1)--(2,2)--(3, 3); 
\gridpoints{3}{3}
\end{tikzpicture},
\begin{tikzpicture}[scale=.30,baseline=2ex]
\draw[ blue] (3,1)--(3,2)--(1, 3); 
\draw[ blue] (2,1)--(1,2)--(2, 3); 
\draw[ blue] (1,1)--(2,2)--(3, 3); 
\gridpoints{3}{3}
\end{tikzpicture},
\begin{tikzpicture}[scale=.30,baseline=2ex]
\draw[ blue] (2,1)--(3,2)--(1, 3); 
\draw[ blue] (3,1)--(1,2)--(2, 3); 
\draw[ blue] (1,1)--(2,2)--(3, 3); 
\gridpoints{3}{3}
\end{tikzpicture},
\begin{tikzpicture}[scale=.30,baseline=2ex]
\draw[ blue] (3,1)--(3,2)--(1, 3); 
\draw[ blue] (1,1)--(1,2)--(2, 3); 
\draw[ blue] (2,1)--(2,2)--(3, 3); 
\gridpoints{3}{3}
\end{tikzpicture}
\end{minipage}\!
 \right\}$.
\end{example}

Reticulations with  parameters $(m\geq 3,n\geq 3,k\geq 2,\ell\geq 2)$ are  challenging
to study.  The maximality of the reticulations constructed in this section sidesteps issues around parastrophic and isotopic ordered reticulations. 

\section{Prolongation}
\label{sec:prolongation}

We adapt a recursive construction from Latin squares \cite[Section 4.3]{MR1096296},  \cite[Section 1.5]{MR3495977} to increase one of the dimensions of a reticulation/cooperative system by one.  
We begin with an analog for the notion of a transversal of a Latin square and the prolongation construction (see \cite[Section 1.5]{MR3495977}). 

\begin{definition}
Let $\mathcal{R}=(\mathcal{P}, \mathcal{F}_{\mathrm{weft}},  \mathcal{F}_{\mathrm{warp}})$ be a reticulation. 
By a \emph{weft semitransversal}, we mean a subset ${T}$ of points such that every weft line is incident to exactly one element of ${T}$.
By a \emph{warp semitransversal}, we mean a subset ${S}$ of points such that every warp line is incident to exactly one element of ${S}$.
\end{definition}

The corresponding notions for cooperative systems offer a more concrete approach.

\begin{definition}
Let  $\mathcal{C}=(C^1, \ldots, C^k; R^1, \ldots, R^\ell)$ be a cooperative system with parameters $(m,n,k,\ell)$.
\begin{enumerate}
\item A \emph{row semitransversal} of $\mathcal{C}$ is a subset ${T}$ of $[m]\times[n]$ such that 
every element of $[m]$ 
      occurs exactly once as a first component of an element of ${T}$ and 
      occurs exactly once among the positions of each $C^u$ $(u\in[k])$ indexed by the elements of ${T}$.
\item A \emph{column semitransversal} of $\mathcal{C}$ is a subset ${S}$ of $[m]\times[n]$ such that 
every element of $[n]$ 
      occurs exactly once as a second component of an element of ${S}$ and 
      occurs exactly once among the positions of each $R^v$ $(v\in[\ell])$ indexed by the elements of ${S}$.      
      
\end{enumerate}     
\end{definition}

For a normalized cooperative system, the two conditions for a row semitransversal are redundant.
The notions of row and column semitransversals are weaker than the usual notion of  transversal for Latin squares \cite{MR1096296} or for row-Latin rectangles \cite{MR1652837, MR387083}, 
as they do not require distinct columns and rows, respectively.

\begin{lemma}
\label{lem:transversalsofCoop}
Let  $\mathcal{C}=(C^1, \ldots, C^k; R^1, \ldots, R^\ell)$ be a cooperative system with parameters $(m,n,k,\ell)$.
\begin{enumerate}
\item For all $v\in[k]$, the set $\{ \{(i,j) :  (i,j)\in[m]\times[n], R^v(i,j)=q  \}_{q\in[n]} \}$ (the set of warp lines associated with $R^v$) is a partition of $[m]\times[n]$ into disjoint row semitransversals.

\item  For all $u\in[\ell]$, the set $\{ \{(i,j) :  (i,j)\in[m]\times[n], C^u(i,j)=r  \}_{r\in[m]} \}$  (the set of weft lines associated with $C^u$)  is a partition of $[m]\times[n]$ into disjoint column semitransversals.   
\end{enumerate}
\end{lemma}

\begin{proof} 
Since $\mathcal{C}$ is a cooperative system, no two pairs have the same first component.  
Since each $C^u$ is column-Latin,  every value in $[m]$ occurs exactly one in column of each $C^u$. Thus, (i) holds. A similar argument gives (ii).
\end{proof}

We are ready to describe a recursive construction of cooperative systems.

\begin{theorem}
Let  $\mathcal{C}=(C^1, \ldots, C^k; R^v, \ldots, R^\ell)$ be a cooperative system with parameters $(m,n,k,\ell)$.
\begin{enumerate}
\item Let $T$ be a row semitransversal of $\mathcal{C}$.
Define $m\times (n+1)$ matrices $\cpro{T}{C}^u$ $(u\in[k])$ and $\cpro{T}{R}^v$ $(v\in[\ell])$ with $(i,j)$-entries
\begin{align}
\cpro{T}{C}^u({i,j}) &= \begin{cases}
                               C^u({i,j}) & \hbox{if $(i,j)\in[m]\times[n]$,}\\
                               C^u({i,\ell})& \hbox{if $i\in[m]$, $j=n+1$, and $(i,\ell)\in T$,}
                            \end{cases} \nonumber \\
\cpro{T}{R}^v({i,j}) &= \begin{cases}
                               R^v({i,j}) & \hbox{if $(i,j)\in([m]\times[n])\setminus T$,}\\
                               n+1 & \hbox{if $(i,j)\in T$,}\\
                               R^v({i,\ell}) & \hbox{if $i\in[m]$, $j=n+1$, and $(i,\ell)\in T$.}
                            \end{cases} \nonumber
\end{align}                                    
Then $\cpro{T}{\mathcal{C}}:=(\cpro{T}{C}^1, \ldots, \cpro{T}{C}^k; \cpro{T}{R}^1, \ldots, \cpro{T}{R}^\ell)$
is a cooperative system with parameters $(m,n+1,k,\ell)$.  We refer to $\cpro{T}{\mathcal{C}}$ as the \emph{column prolongation of $\mathcal{C}$ relative to the row semitransversal $T$}.

\item
Let ${S}$ be column semitransversal of $\mathcal{C}$.  Then ${S}^\dag$ is a row semitransversal of $\mathcal{C}^\dag$.   
Furthermore, $\rpro{{S}}{\mathcal{C}} := (\cpro{{S}^\dag}{\mathcal{C}}^\dag)^\dag$ is a cooperative system with parameters $(m+1,n,k,\ell)$.  
We refer to $\rpro{\mathcal{S}}{\mathcal{C}}$ as the \emph{row prolongation of $\mathcal{C}$ relative to the row semitransversal ${S}$}.
\end{enumerate}
\end{theorem}

\begin{proof} 
By construction, each $\cpro{T}{C}^u$ is column-Latin as the appended column repeats all values in $[m]$ as its entries.
By construction, each $\cpro{T}{R}^v$ is row-Latin as a new entry $n+1$ is added, the previous entry is bumped to the end.
In the locations indexed by $T$, each value of $[m]$ in each $\cpro{T}{C}^u$ is paired with the entry  $n+1$ in  $\cpro{T}{R}^v$.  
The pairs from $[m]\times [n]$ that appeared in the position of $C^u$ and $R^v$ indexed by $T$ appear in the added column.
\end{proof}

The operator \rotatebox{90}{$\ominus$} reminds us that we are placing a new column next to the reticulation, and the operator $\ominus$ reminds us that we are stacking a new row below the  reticulation. 
A row semitransversal of $\mathcal{C}$ is a row semitransversal of $\cpro{T}{\mathcal{C}}$. 
Up to the ordering of the warp lines, repeating the column prolongation construction relative to the same row semitransversal is a splicing horizontal segments to the weft lines and adding new vertical lines to the warp families. 

\begin{example}
We twice prolongate a small cooperative system relative to the row semitransversal marked in boxes.
\[\addtolength{\arraycolsep}{-0.35em}
 {C}^1 =\left[\begin{array}{cccc}
1&1&\fbox{1}&1\\[-1pt]
2&2&2&\fbox{2}\\[-1pt]
\fbox{3}&3&3&3
\end{array}\right], \ 
\cpro{T}{C}^1 =\left[\begin{array}{ccccc}
1&1&\fbox{1}&1 &1\\[-1pt]
2&2&2&\fbox{2}&2\\[-1pt]
\fbox{3}&3&3&3 &3
\end{array}\right],\,
\cpro{T}{\cpro{T}{C}}^1 =\left[\begin{array}{cccccc}
1&1&\fbox{1}&1 &1&1\\[-1pt]
2&2&2&\fbox{2}&2&2\\[-1pt]
\fbox{3}&3&3&3 &3 &3
\end{array}\right], \]
\[ \addtolength{\arraycolsep}{-0.35em}
{C}^2 =\left[\begin{array}{ccccc}
2&3&\fbox{2}&3\\[-1pt]
1&2&3&\fbox{1}\\[-1pt]
\fbox{3}&1&1&2
\end{array}\right],\ 
\cpro{T}{C}^2 =\left[\begin{array}{ccccc}
2&3&\fbox{2}&3&2\\[-1pt]
1&2&3&\fbox{1}&1\\[-1pt]
\fbox{3}&1&1&2&3
\end{array}\right],\ 
\cpro{T}{\cpro{T}{C}}^2 =\left[\begin{array}{cccccc}
2&3&\fbox{2}&3&2&2\\[-1pt]
1&2&3&\fbox{1}&1&1\\[-1pt]
\fbox{3}&1&1&2&3 &3
\end{array}\right],\ 
\]
\[\addtolength{\arraycolsep}{-0.35em}
 {R}^1 =\left[\begin{array}{ccccc}
1&2&\fbox{3}&4\\[-1pt]
1&2&3&\fbox{4}\\[-1pt]
\fbox{1}&2&3&4
\end{array}\right],\ 
\cpro{T}{R}^1 =\left[\begin{array}{ccccc}
1&2&\fbox{5}&4&3\\[-1pt]
1&2&3&\fbox{5}&4\\[-1pt]
\fbox{5}&2&3&4&1
\end{array}\right],\ 
\cpro{T}{\cpro{T}{R}}^1 =\left[\begin{array}{cccccc}
1&2&\fbox{6}&4&3&5\\[-1pt]
1&2&3&\fbox{6}&4&5\\[-1pt]
\fbox{6}&2&3&4&1 &5
\end{array}\right],
\]
\[\addtolength{\arraycolsep}{-0.35em}
 {R}^2 =\left[\begin{array}{cccc}
1&4&\fbox{3}&2\\[-1pt]
4&2&1&\fbox{3}\\[-1pt]
\fbox{3}&1&2&4\\
\end{array}\right],\,
\cpro{T}{R}^2 =\left[\begin{array}{ccccc}
1&4&\fbox{5}&2&3\\[-1pt]
4&2&1&\fbox{5}&3\\[-1pt]
\fbox{5}&1&2&4&3
\end{array}\right],\ 
\cpro{T}{\cpro{T}{R}}^2 =\left[\begin{array}{cccccc}
1&4&\fbox{6}&2&3&5\\[-1pt]
4&2&1&\fbox{6}&3&5\\[-1pt]
\fbox{6}&1&2&4&3&5
\end{array}\right].\] 
\end{example}

\begin{lemma} 
Let $\mathcal{C}$ be a cooperative system, and let $T$ be a row semitransversal of $\mathcal{C}$.
If $\mathcal{C}$ is repetition-free, then so is $\cpro{T}{\mathcal{C}}$.
\end{lemma}

\begin{proof} 
The differences between the column-Latin matrices (weft families) are not altered by adding a new column.
The between the row-Latin matrices (warp families) are either preserved or copied into the added column.
\end{proof} 

The prolongation construction leads to special circumstances under which a cooperative pair can be abbreviated.
Let $\mathcal{C}$ be a cooperative pair with parameters $(m,n,k,\ell)$.  
Let $T$ be a row semitransversal of $\mathcal{C}$.
Suppose that there is a column $j\in[n]$ disjoint from $T$ such that  the following hold.
\begin{enumerate}
\item For all $u\in[k]$, 
        $C^u(i,j)$ is equal to the entry of $C^u$ in the position in row $i$ of  $T$ for all $i\in[m]$.
\item For all $v\in[\ell]$, every entry in column $j$ of $R^v$ is $m$. 
\end{enumerate}
Then the result of replacing $R^v(q,r)$ with the entry in $R^v(q,j)$ for $(q,r)\in T$ $(v\in[\ell])$ and
deleting column $j$ from all $C^u$ and $R^v$  yields a cooperative system with parameter $(m,n-1,k,\ell)$.
Note that in (ii), the constant value need not be $m$ (apply an isotopism to change the value).

\section{Splicing}
\label{sec:splice}

We present a construction of a larger reticulation from smaller reticulations.

\begin{theorem}
\label{thm:splice}
Let $\hat{\mathcal{P}}$ and $\check{\mathcal{P}}$ be disjoint sets.
Suppose 
\begin{eqnarray*}
\hat{\mathcal{R}}&=&(\hat{\mathcal{P}},
  \{\hat{\mathcal{E}}^{\hat{u}}=\{ \hat{E}^{\hat{u}}_{\hat{\imath}}\}_{\hat{\imath}\in[\hat{m}]}\}_{\hat{u}\in[\hat{k}]}, 
  \{\hat{\mathcal{A}}^{\hat{v}}=\{ \hat{A}^{\hat{v}}_{\hat{\jmath}}\}_{\hat{\jmath}\in[\hat{n}]} \}_{\hat{v}\in[\hat{\ell}]}) \ \hbox{and}\\
\check{\mathcal{R}}&=&(\check{\mathcal{P}},
  \{\check{\mathcal{E}}^{\check{u}}=\{ \check{E}^{\check{u}}_{\check{\imath}}\}_{\check{\imath}\in[\check{m}]} \}_{\check{u}\in[\check{k}]}, 
  \{\check{\mathcal{A}}^{\check{v}}=\{ \check{A}^{\check{v}}_{\check{\jmath}}\}_{\check{\jmath}\in[\check{n}]}\}_{\check{v}\in[\check{\ell}]})  
\end{eqnarray*}
are ordered reticulations with parameters 
$(\hat{m}, \hat{n}, \hat{k}, \hat{\ell})$ and $(\check{m}, \check{n}, \check{k}, \check{\ell})$, respectively.    
\begin{enumerate}
\item  \label{thm:weftsplice}
Suppose $\hat{m}=\check{m}$, and let $\mathcal{P}= \hat{\mathcal{P}}\cup\check{\mathcal{P}}$.
\begin{enumerate}
\item For each $\hat{u}\in[\hat{k}]$,   $\check{u}\in[\check{k}]$, $\sigma\in S_{\hat{m}}$, and $i\in[\hat{m}]$, 
         let ${E}^{\hat{u},\check{u},\sigma}_{i} = E^{\hat{u}}_{i}\cup \check{E}^{\check{u}}_{\sigma(i)}$, and set 
      $\mathcal{E}^{\hat{u},\check{u},\sigma} = \{{E}^{\hat{u},\check{u},\sigma}_{i} \}_{i\in[\hat{m}]}$.
         Then the lines in $\mathcal{E}^{\hat{u},\check{u},\sigma}$ partition $\mathcal{P}$.
\item  For each $\hat{v}\in[\hat{\ell}]$, $\check{v}\in[\check{\ell}]$, 
          let  $\mathcal{A}^{\hat{v},\check{v}} =\hat{\mathcal{A}}^{\hat{v}}\cup \check{\mathcal{A}}^{\check{v}}$.
          Then the lines in  $\mathcal{A}^{\hat{v},\check{v}}$ partition $\mathcal{P}$. 
\item For any $\hat{u}\in[\hat{k}]$,   $\check{u}\in[\check{k}]$, $\sigma\in S_{\hat{m}}$ and 
         for any $\hat{v}\in[\hat{\ell}]$, $\check{v}\in[\check{\ell}]$,
         each line in $\mathcal{E}^{\hat{u},\check{u},\sigma}$ meets 
         each line in $\mathcal{A}^{\hat{v},\check{v}}$ in a unique point. 
         Specifically,  
           ${E}^{\hat{u},\check{u},\sigma}_{i}\cap \hat{A}^{\hat{v}}_{\hat{\jmath}} 
                     = \hat{E}^{\hat{u}}_{i} \cap  \hat{A}^{\hat{v}}_{\hat{\jmath}}$  and 
           ${E}^{\hat{u},\check{u},\sigma}_{i} \cap \check{A}^{\check{v}}_{\check{\jmath}}
                     = \check{E}^{\check{u}}_{\sigma(i)}\cap \check{A}^{\check{v}}_{\check{\jmath}}$. 
\item Let $\hat{\mathcal{R}}\boxbar\check{\mathcal{R}}=(
       \mathcal{P}, 
        \{ \mathcal{E}^{\hat{u},\check{u},\sigma}\}_{\hat{u}\in[\hat{k}], \check{u}\in[\check{k}], \sigma\in S_{\hat{m}}},
        \{\mathcal{A}^{\hat{v},\check{v}}\}_{\hat{v}\in[\hat{\ell}], \check{v}\in[\check{\ell}]})$.
      Then $\hat{\mathcal{R}}\boxbar\check{\mathcal{R}}$    is a reticulation with parameters 
             $(\hat{m}, \hat{n}+\check{n}, \hat{k}\cdot \check{k}\cdot \hat{m}!, \hat{\ell}\cdot \check{\ell})$.
       We call $\hat{\mathcal{R}}\boxbar\check{\mathcal{R}}$ the \emph{weft splice} of    $\hat{\mathcal{R}}$ and $\check{\mathcal{R}}$.

\item If $\hat{\mathcal{R}}$ and $\check{\mathcal{R}}$ are repetition-free, then so is 
$\hat{\mathcal{R}}\boxbar\check{\mathcal{R}}$.
\end{enumerate}

\item  \label{thm:warpsplice}
    Suppose $\hat{n}=\check{n}$.  Write $\hat{\mathcal{R}}\boxminus\check{\mathcal{R}}$ to denote
    $(\hat{\mathcal{R}}^\dag\boxbar\check{\mathcal{R}}^\dag)^\dag$.
    Then $\hat{\mathcal{R}}\boxminus\check{\mathcal{R}}$  is a reticulation with parameters $(\hat{m}+\check{m}, \hat{n}, \hat{k}\cdot \check{k}, \hat{\ell}\cdot \check{\ell}\cdot \hat{n}!)$. 
          we call $\hat{\mathcal{R}}\boxminus\check{\mathcal{R}}$ the \emph{warp splice} of    $\hat{\mathcal{R}}$ and $\check{\mathcal{R}}$.
\end{enumerate}
\end{theorem}

\begin{proof}
We prove (i).  Part (ii) follows from Lemma \ref{lem:transposereticulation}.

\noindent
(ia): For fixed $\hat{u}\in[\hat{k}]$, each point $\hat{P}\in \hat{\mathcal{P}}$ lies in exactly one of $\hat{E}^{\hat{u}}_{1}$, $\ldots$, $E^{\hat{u}}_{\hat{m}}$; 
         if it lies in $E^{\hat{u}}_{\hat{\imath}}$, then it also lies in $\hat{E}^{\hat{u},\check{u},\sigma}_{\hat{\imath}}$ for each $\check{u}\in[\check{k}]$.
     For fixed $\hat{u}\in[\hat{k}]$, each point $\check{P}\in \check{\mathcal{P}}$ lies in exactly one of $\check{E}^{\check{u}}_{1}$, $\ldots$, $\check{E}^{\check{u}}_{\check{m}}$; 
         if it lies in $\check{E}^{\check{u}}_{\check{\imath}}$, then it also lies in $\hat{E}^{\hat{u},\check{ s},\sigma}_{\sigma^{-1}(\check{\imath})}$ for each $\hat{u}\in[\hat{k}]$.
Thus, every element of $\hat{\mathcal{P}}\cup\check{\mathcal{P}}$ is in at least one of the lines in $\mathcal{E}^{\hat{u},\check{u},\sigma}$.  
These lines are disjoint by construction since distinct  $\hat{E}^{\hat{u}}_{\ell}$  are disjoint and distinct $\check{E}^{\check{u}}_{\sigma(\ell)}$ are disjoint.
     
\noindent
(ib): Each point $\hat{P}\in\hat{\mathcal{P}}$ lies in exactly one of $\hat{A}^{\hat{v}}_{1}$, $\ldots$, $\hat{A}^{\hat{v}}_{\hat{n}}$ in both $\hat{\mathcal{R}}$ and  $\mathcal{R}$.
      Each point   $\check{P}\in \check{\mathcal{P}}$ lies in exactly one of $\check{A}^{\check{v}}_{1}$, $\ldots$, $\check{A}^{\check{v}}_{\check{n}}$ in both  $\check{\mathcal{R}}$ and $\mathcal{R}$. 
      Thus,  $\mathcal{A}^{\hat{v},\check{v}}$ is a partition of  $\mathcal{P}$.

\noindent     
(ic): The intersections follow from the construction.  
      Note that  $|\hat{E}^{\hat{u}}_{i} \cap  \hat{A}^{\hat{v}}_{\hat{\jmath}} |=1$ and $|\check{E}^{\check{u}}_{\sigma(i)}\cap \check{A}^{\check{v}}_{\check{\jmath}}|=1$ 
      since  $\hat{\mathcal{R}}$  and  $\check{\mathcal{R}}$ are reticulations.

\noindent
(id): Clear from Definition \ref{def:reticulation} and parts (i)--(iii).
 
 \noindent
(ie):   If ${R}$ is not repetition-free because some  ${\mathcal{A}}^{\hat{\jmath},\check{\jmath}}$ and  ${\mathcal{A}}^{\hat{v},\check{v}}$ are equal for distinct $(\hat{\jmath},\check{\jmath})$ and $(\hat{v},\check{v})$, 
         then by construction, both $\hat{\mathcal{A}}^{\hat{\jmath}}=\hat{\mathcal{A}}^{\hat{v}}$ and  $\check{\mathcal{A}}^{\check{\jmath}}=\check{\mathcal{A}}^{\check{v}}$. 
         So, one of $\hat{\mathcal{R}}$ and $\check{\mathcal{R}}$ has a repeated warp family. 
         If $\mathcal{R}$ is not repetition-free because some 
         $\mathcal{E}^{\hat{\imath},\check{\imath},\sigma} =  \mathcal{E}^{\hat{u},\check{u},\tau}$ for some distinct $(\hat{\imath},\check{\imath},\sigma)$ and $(\hat{u},\check{u},\tau)$.
         By construction $\hat{\mathcal{E}}^{\hat{\imath}}=\hat{\mathcal{E}}^{\hat{u}}$, $\check{\mathcal{E}}^{\sigma(\check{\imath})}=\check{\mathcal{E}}^{\tau(\check{u})}$. 
         If  $\hat{i}\not=\hat{u}$, then $\hat{\mathcal{R}}$ is not repetition-free, so suppose  $\hat{i}=\hat{u}$. Given this, were $\sigma\not=\tau$, then $\check{\mathcal{E}}^{\sigma(\check{\imath})}\not=\check{\mathcal{E}}^{\tau(\check{u})}$.   
         Thus, we may assume $\sigma=\tau$.  Now, this would force $\check{\imath}\not=\check{u}$, so $\check{\mathcal{R}}$ is not repetition-free. Thus, $\mathcal{R}$ is repetition-free.
\end{proof}

The operator $\boxbar$ reminds us that we are placing the reticulations side-by-side, and the operator $\boxminus$ reminds us that we are stacking the first reticulation over the second. 

 \begin{example}\label{ex:splice}
 We offer a small example of warp splicing. 
 
\noindent
$\left(\!\!\!\left\
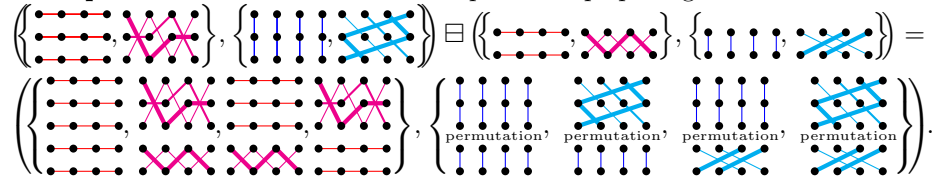
{
\begin{tikzpicture}[scale=.30,baseline=3.5ex]
\foreach \x in {1,2,3, 4}. \foreach \y in {1,2, 3}
{ \draw[ red] (1,\y)--(4,\y); };
\gridpoints{3}{4}
\end{tikzpicture},
\begin{tikzpicture}[scale=.30,baseline=3.5ex]
\draw[thick,magenta] (1,2)--(2,2) -- (3, 3) -- (4,1) ;
\draw[magenta] (1,1)--(2,3)--(3,1)--(4,3); 
\draw[ultra thick, magenta]  (1,3)--(2,1)--(3, 2)--(4, 2);
\gridpoints{3}{4}
\end{tikzpicture}
\right\}\!, 
\left\{
 \begin{tikzpicture}[scale=.30,baseline=3.5ex]
\foreach \x in {1,2,3, 4}. \foreach \y in {1,2, 3}
{ \draw[ blue] (\x,1)--(\x, 3); };
\gridpoints{3}{4}
\end{tikzpicture},
\begin{tikzpicture}[scale=.30,baseline=3.5ex]
\draw[ultra thick, cyan] (4,3)--(1,2)--(2,1);
\draw[cyan] (1,3)--(3,1);
\draw[thick, cyan] (2,3)--(4,1);
\draw[ultra thick, cyan] (1,1)--(4,2)--(3,3);
\gridpoints{3}{4}
\end{tikzpicture}
\right\}\!\!\!\right)$
$\!\boxminus\!$
$\left(\!\!\left\{
\begin{tikzpicture}[scale=.30,baseline=3.1ex]
\foreach \y in {1,2}
{ \draw[ red] (1,\y)--(4,\y); };
\gridpoints{2}{4}
\end{tikzpicture},
\begin{tikzpicture}[scale=.30,baseline=3.1ex]
\draw[magenta] (1,1)--(2,2)--(3,1)--(4,2); 
\draw[ultra thick, magenta]  (1,2)--(2,1)--(3, 2)--(4, 1);
\gridpoints{2}{4}
\end{tikzpicture}
\right\}\!, 
\left\{
 \begin{tikzpicture}[scale=.30,baseline=3.1ex]
\foreach \x in {1,2,3, 4}
{ \draw[ blue] (\x,1)--(\x, 2); };
\gridpoints{2}{4}
\end{tikzpicture},\,
\begin{tikzpicture}[scale=.30,baseline=3.1ex]
\draw[ cyan] (1,2)--(3,1);
\draw[ cyan] (2,2)--(4,1);
\draw[ultra thick, cyan] (1,1)--(3,2);
\draw[ultra thick, cyan] (4,2)--(2,1);
\gridpoints{2}{4}
\end{tikzpicture}
\right\}\!\!\right)=$

\smallskip
\noindent
$\left(\!\!\left\{
\begin{tikzpicture}[scale=.30,baseline=23]
\foreach \y in {1,2,3,4,5}
{ \draw[ red] (1,\y)--(4,\y); };
\gridpoints{5}{4}
\end{tikzpicture},\,
\begin{tikzpicture}[scale=.30,baseline=23]
\draw[magenta] (1,1)--(2,2)--(3,1)--(4,2); 
\draw[ultra thick, magenta]  (1,2)--(2,1)--(3, 2)--(4, 1);

\draw[thick,magenta] (1,4)--(2,4) -- (3, 5) -- (4,3) ;
\draw[magenta] (1,3)--(2,5)--(3,3)--(4,5); 
\draw[ultra thick, magenta]  (1,5)--(2,3)--(3, 4)--(4, 4);

\gridpoints{5}{4}
\end{tikzpicture},
\begin{tikzpicture}[scale=.30,baseline=23]
\draw[magenta] (1,1)--(2,2)--(3,1)--(4,2); 
\draw[ultra thick, magenta]  (1,2)--(2,1)--(3, 2)--(4, 1);

\foreach \y in {3,4,5}
{ \draw[ red] (1,\y)--(4,\y); };

\gridpoints{5}{4}
\end{tikzpicture},
\begin{tikzpicture}[scale=.30,baseline=23]

\foreach \y in {1,2}
{ \draw[ red] (1,\y)--(4,\y); };

\draw[thick,magenta] (1,4)--(2,4) -- (3, 5) -- (4,3) ;
\draw[magenta] (1,3)--(2,5)--(3,3)--(4,5); 
\draw[ultra thick, magenta]  (1,5)--(2,3)--(3, 4)--(4, 4);

\gridpoints{5}{4}
\end{tikzpicture}
\right\}\!, 
\left\{\!\!\!
 \begin{tikzpicture}[scale=.30,baseline=23]
\foreach \x in {1,2,3, 4}
 { \draw[ blue] (\x,3)--(\x, 5); };
 \node[ ]at (2.5,2.5) {\hbox{\tiny permutation}}; 
\foreach \x in {1,2,3, 4}
{ \draw[ blue] (\x,1)--(\x, 2); };
\gridpoints{5}{4}
\end{tikzpicture}\!\!,
 \begin{tikzpicture}[scale=.30,baseline=23]
\draw[ultra thick, cyan] (4,5)--(1,4)--(2,3);
\draw[cyan] (1,5)--(3,3);
\draw[thick, cyan] (2,5)--(4,3);
\draw[ultra thick, cyan] (1,3)--(4,4)--(3,5);
 \node[ ]at (2.5,2.5) {\hbox{\tiny permutation}}; 
\foreach \x in {1,2,3, 4}
{ \draw[ blue] (\x,1)--(\x, 2); };
\gridpoints{5}{4}
\end{tikzpicture}\!\!,
 \begin{tikzpicture}[scale=.30,baseline=23]
\foreach \x in {1,2,3, 4}
 { \draw[ blue] (\x,3)--(\x, 5); };
 \node[ ]at (2.5,2.5) {\hbox{\tiny permutation}}; 
\draw[ cyan] (1,2)--(3,1);
\draw[ cyan] (2,2)--(4,1);
\draw[ultra thick, cyan] (1,1)--(3,2);
\draw[ultra thick, cyan] (4,2)--(2,1);
\gridpoints{5}{4}
\end{tikzpicture}\!\!,
\begin{tikzpicture}[scale=.30,baseline=23]
\draw[ultra thick, cyan] (4,5)--(1,4)--(2,3);
\draw[cyan] (1,5)--(3,3);
\draw[thick, cyan] (2,5)--(4,3);
\draw[ultra thick, cyan] (1,3)--(4,4)--(3,5);
 \node[ ]at (2.5,2.5) {\hbox{\tiny permutation}}; 
\draw[ cyan] (1,2)--(3,1);
\draw[ cyan] (2,2)--(4,1);
\draw[ultra thick, cyan] (1,1)--(3,2);
\draw[ultra thick, cyan] (4,2)--(2,1);
\gridpoints{5}{4}
\end{tikzpicture}
\!\!\!\right\}\!\!\right)$\!.

The permutations can be any bijective correspondence between points the points in the third and fourth rows.  
For instance, take the permutation to be 
\begin{tikzpicture}[scale=.30]
\draw[thick, brown] (1, 1)--(2, 2);
\draw[thick, brown] (2, 1)--(1, 2);
\draw[thick, brown] (3 ,1)--(4, 2);
\draw[thick, brown] (4, 1)--(3, 2);
\gridpoints{2}{4}
\end{tikzpicture}, or
\begin{tikzpicture}[scale=.30]
\draw[thick, brown] (1, 1)--(2, 2);
\draw[thick, brown] (2, 1)--(3, 2);
\draw[thick, brown] (3, 1)--(4, 2);
\draw[thick, brown] (4, 1)--(1, 2);
\gridpoints{2}{4}
\end{tikzpicture}.
 \end{example}

We immediately get the following.

\begin{corollary}
With reference to Theorem \ref{thm:splice}, 
the cooperative matrices of $\hat{\mathcal{R}}\boxbar\check{\mathcal{R}}$ in block form are as follows.
For  $\hat{u}\in[\hat{k}]$, $\check{u}\in[\check{k}]$,
          $\hat{v}\in[\hat{\ell}]$, $\check{v}\in[\check{\ell}]$
\[ C^{\hat{u},\check{u},\sigma}=\mu(\mathcal{E}^{\hat{u},\check{u},\sigma}) 
        = [\, C^{\hat{u}} \ \sigma(C^{\check{u}})\, ] \quad \hbox{and}\quad
   A^{\hat{v},\check{v}} =\mu(\mathcal{A}^{\hat{v},\check{v}}) =  [\, R^{\hat{u}}  \ R^{\check{u}}  + \hat{n}\,],
\]
where 
\begin{align*}
   C^{\hat{u}} &= \mu(\mathcal{E}^{\hat{u}}), \qquad &
    C^{\check{u}} &= \mu(\mathcal{E}^{\check{u}}),  &
      \sigma(C^{\check{u}})(i,j) &= \sigma(C^{\check{u}}(i,j)),\\
    R^{\hat{u}} &= \mu(\mathcal{A}^{\hat{u}}),\qquad &
    R^{\check{u}} &= \mu(\mathcal{E}^{\check{u}}), &
    (R^{\check{u}} + \hat{n})(i,j) &= R^{\check{u}}(i,j)  + \hat{n}.
\end{align*}
\end{corollary}

\begin{definition}
We say that a reticulation is \emph{weft reducible} if it is contained in the weft splice of two reticulations with nonempty warp families and that it is \emph{weft reduced}  otherwise.  
We say that a reticulation is \emph{warp reducible}  if it is contained in the warp splice of two reticulations with nonempty weft families and that it is \emph{warp reduced}  otherwise.
We say that a reticulation is \emph{reducible} if it is either weft and warp reducible, and that it is 
\emph{reduced} if it is both weft and warp reduced.
\end{definition}

\begin{theorem}
A reticulation is weft (warp) reducible if and only if there is a  partition of the points into two parts so that all points on any given warp (weft) line lie entirely in one of the cells of the partition. 
\end{theorem}

\begin{proof} 
If a reticulation $\mathcal{R}$ is weft reducible, then it is contained in the weft splice $\hat{\mathcal{R}}\boxbar\check{\mathcal{R}}$ of two reticulations.  
By construction (Theorem  \ref{thm:splice}), all warp lines consist of either points from $\hat{\mathcal{R}}$ or points from $\check{\mathcal{R}}$.  Thus, the forward implication holds.

Conversely, suppose that there is a  partition of the points into two parts $\mathcal{P}=\hat{\mathcal{P}}\cup\check{\mathcal{P}}$ so that all points on any given warp (weft) line lie entirely in one of the cells of the partition. 
For each weft family $\mathcal{E}^u$ of $\mathcal{R}$, let $\hat{\mathcal{E}}^u$ and $\check{\mathcal{E}}^u$  be the restrictions of each line in $\mathcal{E}^u$ to $\hat{\mathcal{P}}$ and $\check{\mathcal{P}}$, respectively. 
For each warp family, $\mathcal{A}^v$ of $\mathcal{R}$, let $\hat{\mathcal{A}}^v$ and $\check{\mathcal{A}}^v$  be the sets of lines whose points are in $\hat{\mathcal{P}}$ and $\check{\mathcal{P}}$, respectively.  
We claim that $\hat{\mathcal{R}}=(\hat{P}, \{\hat{\mathcal{E}}^u\}, \{\hat{\mathcal{A}}^v\})$ and $\check{\mathcal{R}}=(\check{P}, \{\check{\mathcal{E}}^u\}, \{\check{\mathcal{A}}^v\})$ are reticulations.  
The unique intersection property (R-1) holds for each reticulation since the unique intersection of 
any line from $\mathcal{E}^u$ and any line from $\mathcal{A}^v$ is in either $\hat{P}$ or $\check{P}$.  
The partition property (R-2) holds for each reticulation since the partitions given by $\mathcal{E}^u$ and  $\mathcal{A}^v$ are refined to partitions of $\hat{P}$ and $\check{P}$.  
Now $\mathcal{R}$ is contained in $\hat{\mathcal{R}}\boxbar\check{\mathcal{R}}$ by construction.  Note that repetition may be introduced by this construction. 
\end{proof}

\begin{lemma}
With reference to Theorem \ref{thm:splice},
 $\hat{\mathcal{R}}\boxbar\check{\mathcal{R}}$ and $\check{\mathcal{R}}\boxbar\hat{\mathcal{R}}$ are isotopic, and $\hat{\mathcal{R}}\boxminus\check{\mathcal{R}}$ and $\check{\mathcal{R}}\boxminus\hat{\mathcal{R}}$ are isotopic. 
\end{lemma}

\begin{proof} 
The columns of  $\hat{\mathcal{R}}\boxbar\check{\mathcal{R}}$ can be permuted to get the columns of $\check{\mathcal{R}}\boxbar\hat{\mathcal{R}}$.
\end{proof} 

\section{A direct product}
\label{sec:blockproduct}

We describe a construction of a reticulation from two smaller reticulations in which each point in a family of weft/warp lines in the first reticulation is replaced with a family of lines of the same type from the second reticulation to form a new family of lines.  
This product is analogous to the direct or Kronecker product of mutually orthogonal Latin squares \cite[Section 4.3]{MR1096296}, \cite[Lemma 1.8]{MR3837138}, \cite[Theorem 2.6]{MR1644242}.

\begin{theorem}
\label{thm:blockproduct}
Let 
$\hat{\mathcal{R}}=(\hat{\mathcal{P}},\hat{\mathcal{F}}_{\mathrm{weft}},\hat{\mathcal{F}}_{\mathrm{warp}})$
and 
$\check{\mathcal{R}}=(\check{\mathcal{P}},\check{\mathcal{F}}_{\mathrm{weft}},\check{\mathcal{F}}_{\mathrm{warp}})$
be ordered reticulations with respective parameters $(\hat{m},\hat{n},\hat{k},\hat{\ell})$ and $(\check{m}, \check{n}, \check{k}, \check{\ell})$. Construct the following.
\begin{enumerate}
\item  Let $\mathcal{P}=\hat{\mathcal{P}}\times\check{\mathcal{P}}$ (Cartesian product).
\item For each $\hat{\mathcal{E}}\in \hat{\mathcal{F}}_{\mathrm{weft}}$, and
         for each function $\phi:\hat{\mathcal{P}}\rightarrow  \check{\mathcal{F}}_{\mathrm{weft}}$,
         form a family of lines $\mathcal{E}^{\hat{\mathcal{E}}, \phi}$ on $\mathcal{P}$ as follows:
         $(\hat{P}, \check{P})$ and $(\hat{Q}, \check{Q})\in \mathcal{P}$ are collinear in $\mathcal{E}^{\hat{\mathcal{E}}, \phi}$ 
         whenever $\hat{P}$ and $\hat{Q}$ are collinear in $\hat{\mathcal{E}}$ and
         the respective lines containing $\check{P}$ and $\check{Q}$ in $\phi(\hat{P})$ 
         and $\phi(\hat{Q})$ have the same indices.
 \item For each $\hat{\mathcal{A}}\in \hat{\mathcal{F}}_{\mathrm{warp}}$, and
         for each function $\psi:\hat{\mathcal{P}}\rightarrow  \check{\mathcal{F}}_{\mathrm{warp}}$,
         form a family of lines $\mathcal{A}^{\hat{\mathcal{A}}, \psi}$ on ${\mathcal{P}}$ as follows:
         $(\hat{P}, \check{P})$ and $(\hat{Q}, \check{Q})\in \mathcal{P}$ are collinear in $\mathcal{A}^{\hat{\mathcal{A}}, \psi}$ 
         whenever $\hat{P}$ and $\hat{Q}$ are collinear in $\hat{\mathcal{A}}$ and
         the respective lines containing $\check{P}$ and $\check{Q}$ in $\psi(\hat{P})$ 
         and $\psi(\hat{Q})$ have the same indices.
\end{enumerate}        
Then 
${\mathcal{R}}=
({\mathcal{P}},
        {\mathcal{F}}_{\mathrm{weft}}
            = \{\mathcal{E}^{\hat{\mathcal{E}}, \phi} : \hat{\mathcal{E}}\in \hat{\mathcal{F}}_{\mathrm{weft}}, \phi:\hat{\mathcal{P}}\rightarrow  \check{\mathcal{F}}_{\mathrm{weft}}  \},
         {\mathcal{F}}_{\mathrm{warp}} 
             = \{\mathcal{A}^{\hat{\mathcal{A}}, \psi} :  \hat{\mathcal{A}}\in \hat{\mathcal{F}}_{\mathrm{warp}}, \psi:\hat{\mathcal{P}}\rightarrow  \check{\mathcal{F}}_{\mathrm{warp}}\})$
 is a reticulation with parameters 
    $(\hat{m}\check{m}, \hat{n}\check{n}, \hat{k}\check{k}^{\hat{m}\hat{n}}, \hat{\ell}\check{\ell}^{\hat{m}\hat{n}})$.   
    We refer to ${\mathcal{R}}$ as the \emph{direct product} of $\hat{\mathcal{R}}$ and $\check{\mathcal{R}}$ and denote it by  ${\mathcal{R}}= \hat{\mathcal{R}}\otimes \check{\mathcal{R}}$.
\end{theorem}
       
\begin{proof} 
By construction, the lines in 
     each $\mathcal{E}^{\hat{\mathcal{E}}, \phi}$ and 
     each $\mathcal{A}^{\hat{\mathcal{A}}, \psi}$
partition ${P}$, so (R-2) holds for $\mathcal{R}$.  
Fix lines 
      $E\in \mathcal{E}^{\hat{\mathcal{E}}, \phi}$ and 
      $A\in\mathcal{A}^{\hat{\mathcal{A}}, \psi}$.
The restriction of $E$ and $A$ to $\hat{\mathcal{P}}$ are lines 
      $\hat{E}\in\hat{\mathcal{E}}$ and 
      $\hat{A}\in\hat{\mathcal{A}}$.  
Let $\hat{P}$ be the unique intersection of $\hat{\mathcal{E}}$ and $\hat{\mathcal{A}}$, and 
let $\check{P}$ be the unique intersection of $\phi(\hat{P})$ and $\psi(\hat{P})$.   
Then $(\hat{P},\check{P})$ is in the intersection of 
      ${E}$ and 
      each ${A}$.  
Suppose $(\hat{Q},\check{Q})$ is in the intersection of these two lines.
Then $\hat{P}$ and $\hat{Q}$ are  both 
      the same line of $\hat{\mathcal{E}}$ and 
      the same line of $\hat{\mathcal{A}}$.  
Thus, $\hat{P}=\hat{Q}$ by (R-1).  
Now 
      $\phi(\hat{P})=\phi(\hat{Q})$ and 
      $\psi(\hat{P})=\psi(\hat{Q})$, 
so $\check{P}$ and $\check{Q}$ are in the same line of $\phi(\hat{P})$ and the same line of $\psi(\hat{P})$.  
But by (R-1),  $\check{P}$ is the unique intersection of these lines, so $\check{P}=\check{Q}$.  
Hence, (R-1) holds for $\mathcal{R}$. Thus, ${\mathcal{R}}$ is a reticulation.   
The parameters of  ${\mathcal{R}}$ follow from the construction. 
\end{proof}

We describe the cooperative matrices of the direct product.  
The entries of the column- and row-Latin matrices will be taken to be pairs from $[\hat{m}]\times[\check{m}]$ and $[\hat{n}]\times[\check{n}]$, respectively, 
since these are indices for the lines.  
They are formed by a generalized Kronecker product similar to that of regular combinatorial structures. 

\begin{corollary}
With reference to Theorem \ref{thm:blockproduct}, 
write 
    $\mu(\hat{\mathcal{R}}) = (\hat{C}^1, \hat{C}^2, \ldots, \hat{C}^{\hat{k}};$ 
                                               $\hat{R}^1, \hat{R}^2, \ldots, \hat{R}^{\hat{\ell}})$ and
    $\mu(\check{\mathcal{R}}) = (\check{C}^{1}, \check{C}^{2}, \ldots, \check{C}^{\check{k}};$ 
                                                    $\check{R}^{1}, \check{R}^{2}, \ldots, \check{R}^{\check{\ell}})$.
Then for $\hat{u}\in[\hat{k}]$, $\phi:\hat{P}\rightarrow[\check{k}]$, and 
\begin{align*}
   \mu(\mathcal{E}^{\hat{\mathcal{E}}^{\hat{u}}, \phi})({i,j})
         &= (\hat{C}^{\hat{u}}({\hat{\imath},\hat{\jmath}}), \check{C}({\check{\imath},\check{\jmath}})),
          \quad\hbox{ where $i= \hat{m}(\hat{\imath}-1)+\check{\imath}$ and 
                                         $j= \hat{n}(\hat{\jmath}-1)+\check{\jmath}$,}\\
   \mu(\mathcal{A}^{\hat{\mathcal{A}}^{\hat{v}}, \psi})({i,j}) 
         &= (\hat{R}^{\hat{v}}({\hat{\imath},\hat{\jmath}}), \check{R}({\check{\imath},\check{\jmath}})),
          \quad\hbox{ where $i= \hat{m}(\hat{\imath}-1)+\check{\imath}$ and 
                                         $j= \hat{n}(\hat{\jmath}-1)+\check{\jmath}$.}
\end{align*}
In particular, 
      ${C}^{\hat{\mathcal{E}}^u, \phi}=\mu'(\mathcal{E}^{\hat{\mathcal{E}}, \phi})$ and 
      ${R}^{\hat{\mathcal{A}}^v, \psi}=\mu'(\mathcal{A}^{\hat{\mathcal{A}}, \psi})$ are 
$\hat{m}\check{m}\times \hat{n}\check{n}$ matrices with blocks of size $\check{m}\times \check{n}$, which have blocks
\[
   {C}^{\mathcal{E}^u, \phi} = (\hat{C}^u_{\hat{\imath},\hat{\jmath}}, \phi(P_{\hat{\imath},\hat{\jmath}})), \qquad 
   {R}^{\mathcal{A}^v, \psi} = (\hat{R}^v_{\hat{\imath},\hat{\jmath}}, \psi(P_{\hat{\imath},\hat{\jmath}})).
\]
\end{corollary}

\begin{proof} 
Theorem \ref{thm:blockproduct} describes when two points are collinear in the family  $\mathcal{E}^{\hat{\mathcal{E}}, \phi}$ of weft lines.  
These lines are indexed by $[\hat{m}]\times[\check{m}]$. Two points $(\hat{P}, \check{P})$ and $(\hat{Q}, \check{Q})$ are on the same line precisely when $\hat{P}$ and $\hat{Q}$ are on the same line of $\hat{\mathcal{E}}$ 
(that is, the corresponding entries in $\hat{C}^{\hat{u}}$ are the same) and  the lines containing  $\check{P}$ and $\check{Q}$ in $\phi(\hat{P})$ and $\phi(\hat{Q})$ have the same indices (that is, the corresponding entries in $\check{C}^{\check{u}}$ are the same).
A similar result holds for the warp lines. 
\end{proof} 

\begin{corollary}
With reference to Theorem \ref{thm:blockproduct}, the following hold.
\begin{enumerate}
\item If $\hat{\mathcal{R}}$ and $\check{\mathcal{R}}$ are repetition-free, then so is $\hat{\mathcal{R}}\otimes\check{\mathcal{R}}$.
\item If $\hat{\mathcal{R}}$ and $\check{\mathcal{R}}$ are weft (warp) reduced, then so is $\hat{\mathcal{R}}\otimes\check{\mathcal{R}}$.
\item $\hat{\mathcal{R}}\otimes\check{\mathcal{R}}$ and $\check{\mathcal{R}}\otimes\hat{\mathcal{R}}$ are isotopic.
\end{enumerate}
\end{corollary}

\begin{proof} 
For (i), suppose $\hat{\mathcal{R}}\otimes\check{\mathcal{R}}$ is not repetition-free, and say 
${C}^{\mathcal{E}^u, \phi} = {{C}'}^{{\mathcal{E}}'^{u'}, \phi'}$. Then in the corresponding matrices in first factors
$\hat{C}^{\hat{u}}$ and ${\hat{C}}'^{\hat{u}'}$ are equal by construction, so $\hat{\mathcal{R}}$ is not repetition-free.
A similar argument treats the case where two row-Latin matrices are equal.  For (ii), suppose $\hat{\mathcal{R}}\otimes\check{\mathcal{R}}$ is not weft reduced but $\check{\mathcal{R}}$ is.  Then there is a partition of the point set $\hat{\mathcal{P}}\otimes\check{\mathcal{P}}$ so that the points incident to any weft line are in the same cell of the partition.  Finally, for (iii), the rows and columns of $\hat{\mathcal{R}}\otimes\check{\mathcal{R}}$ can be sorted (via isotopy) into blocks representing $\check{\mathcal{R}}\otimes\hat{\mathcal{R}}$.
\end{proof} 

\begin{example}
The direct product of two  $(2,3,2,2)$-reticulations  has parameters 
$(2\cdot2=4,3\cdot3=9,2\cdot 2^6=128, 2\cdot 2^6=128)$.  
Here, we sample just a few of the families of lines in the block product.

\noindent
\begin{tabular}{|cc|cc||cc|cc|}
\hline
\multicolumn{4}{|c||}{$\hat{\mathcal{R}}$} & \multicolumn{4}{c|}{$\check{\mathcal{R}}$}\\
\hline
$\hat{\mathcal{E}}^1$ & $\hat{\mathcal{E}}^2$ &$\hat{\mathcal{A}}^1$ & $\hat{\mathcal{A}}^2$ &
$\check{\mathcal{E}}^1$ & $\check{\mathcal{E}}^2$ &$\check{\mathcal{A}}^1$ & $\check{\mathcal{A}}^2$ {\phantom{$\displaystyle{\int}$}}\\
\!\!\!
\begin{tikzpicture}[scale=.5]
\foreach \y in {1,2}
{ \draw[ red] (1,\y)--(3,\y); };
\foreach \x in {1,2,3} \foreach \y in {1,2}
   \fill[black] (\x,\y) circle(5pt) ; 
\end{tikzpicture}
\!\!\!\!\!\!
&
\!\!\!\!\!\!
\begin{tikzpicture}[scale=.5]
\draw[magenta] (1,1)--(2,2)--(3,2); 
\draw[ultra thick, magenta]  (1,2)--(2,1)--(3, 1);
\foreach \x in {1,2,3} \foreach \y in {1,2}
   \fill[black] (\x,\y) circle(5pt) ; 
\end{tikzpicture}
\!\!\!\!
&
\!\!\!\!
 \begin{tikzpicture}[scale=.5]
\foreach \x in {1,2,3}
{ \draw[ blue] (\x,1)--(\x, 2); };
\foreach \x in {1,2,3} \foreach \y in {1,2}
   \fill[black] (\x,\y) circle(5pt) ; 
\end{tikzpicture}
\!\!\!\!\!\!
&
\!\!\!\!\!\!
\begin{tikzpicture}[scale=.5]
\draw[ cyan] (1,1)--(1,2);
\draw[ cyan] (2,2)--(3,1);
\draw[ultra thick, cyan] (2,1)--(3,2);
\foreach \x in {1,2,3} \foreach \y in {1,2}
   \fill[black] (\x,\y) circle(5pt) ; 
\end{tikzpicture}
\!\!\!\!\!
&  
\!\!\!\!
\begin{tikzpicture}[scale=.5]
\foreach \y in {1,2}
{ \draw[ red] (1,\y)--(3,\y); };
\foreach \x in {1,2,3} \foreach \y in {1,2}
   \fill[black] (\x,\y) circle(5pt) ; 
\end{tikzpicture}
\!\!\!\!\!\!
&
\!\!\!\!\!\!
\begin{tikzpicture}[scale=.5]
\draw[magenta] (1,1)--(2,1)--(3,2); 
\draw[ultra thick, magenta]  (1,2)--(2,2)--(3, 1);
\foreach \x in {1,2,3} \foreach \y in {1,2}
   \fill[black] (\x,\y) circle(5pt) ; 
\end{tikzpicture}
\!\!\!\!
&
\!\!\!\!
 \begin{tikzpicture}[scale=.5]
\foreach \x in {1,2,3}
{ \draw[ blue] (\x,1)--(\x, 2); };
\foreach \x in {1,2,3} \foreach \y in {1,2}
   \fill[black] (\x,\y) circle(5pt) ; 
\end{tikzpicture}
\!\!\!\!\!\!
&
\!\!\!\!\!\!
\begin{tikzpicture}[scale=.5]
\draw[ cyan] (1,1)--(2,2);
\draw[ cyan] (1,2)--(2,1);
\draw[ultra thick, cyan] (3,1)--(3,2);
\foreach \x in {1,2,3} \foreach \y in {1,2}
   \fill[black] (\x,\y) circle(5pt) ; 
\end{tikzpicture}
\!\!\!\!\!\!
\\ 
\!\!\!\!
$\begin{bmatrix}1&1&1\\2&2&2\end{bmatrix}$
\!\!\!\!\!\!
&
\!\!\!\!\!\!
$\begin{bmatrix}1&2&2\\2&1&1\end{bmatrix}$
\!\!\!\!
&
\!\!\!\!
$\begin{bmatrix}1&2&3\\1&2&3\end{bmatrix}$
\!\!\!\!\!\!
&
\!\!\!\!\!\!
$\begin{bmatrix}1&2&3\\1&3&2\end{bmatrix}$
\!\!\!\!
& 
\!\!\!\!\!
$\begin{bmatrix}1&1&1\\2&2&2\end{bmatrix}$
\!\!\!\!\!\!
&
\!\!\!\!\!\!
$\begin{bmatrix}1&1&2\\2&2&1\end{bmatrix}$
\!\!\!\!
&
\!\!\!\!
$\begin{bmatrix}1&2&3\\1&2&3\end{bmatrix}$
\!\!\!\!\!\!
&
\!\!\!\!\!\!
$\begin{bmatrix}1&2&3\\2&1&3\end{bmatrix}$
\!\!\!\!\!\!
\\

\hline
\end{tabular}

 \setlength{\tabcolsep}{2pt} 

\noindent
\begin{tabular}{|c|c|}
\hline
$\hat{\mathcal{E}}^1$, $\phi=\begin{bmatrix} 1&1&1\\1&1&1\end{bmatrix} $  
&
$\hat{\mathcal{E}}^2$, $\phi=\begin{bmatrix} 1&1&2\\2&1&1\end{bmatrix} $  \\
\hline \!\!\!
$\begin{array}{ccc|ccc|ccc}
11 & 11& 11 & 11& 11& 11 & 11& 11 &11 \\
12 & 12& 12 & 12& 12& 12 & 12& 12 &12\\
\hline
21 & 21& 21 & 21& 21& 21 & 21& 21&21\\
22 & 22& 22 & 22& 22& 22 & 22& 22 & 22
\end{array}$ \!\!\!
& \!\!\!
$\begin{array}{ccc|ccc|ccc}
11& 11& 11& 21& 21& 21& 21&21&22\\
12& 12& 12& 22& 22& 22& 22&22&21\\
\hline
21 & 21& 22& 11&11 & 11& 11& 11&11\\
22 & 22& 21& 12&12 & 12& 12& 12&12
\end{array}$ \!\!\!
\\
\hline \!\!\!
$\hat{\mathcal{A}}^1$, $\psi=\begin{bmatrix} 1&1&1\\1&1&1\end{bmatrix} $  
&
$\hat{\mathcal{A}}^2$, $\psi=\begin{bmatrix} 1&2&1\\2&1&2\end{bmatrix} $  \\
\hline
 \!\!\!
$\begin{array}{ccc|ccc|ccc}
11& 12& 13& 21& 22& 23& 31&32&33\\
11& 12& 13& 21& 22& 23& 31&32&33\\
\hline
11 & 12& 13& 21&22 &23& 31& 32&33\\
11 & 12& 13& 21&22 & 23& 31& 32&33
\end{array}$ \!\!\!
& \!\!\!
$\begin{array}{ccc|ccc|ccc}
11& 12& 13& 21& 22& 23& 31&32&33\\
11& 12& 13& 21& 22& 23& 32&31&33\\
\hline
11 & 12& 13& 31&32 &33& 21& 22&23\\
12 & 11& 13& 31&32 & 33& 21& 22&23
\end{array}$ \!\!\!\\
\hline
\end{tabular}
\end{example}

\end{document}